\documentclass[12pt,reqno]{amsart}

\usepackage{amssymb}
\usepackage{amsmath}
\usepackage{latexsym}
\usepackage{amsthm}
\usepackage{epsfig}
\usepackage{color}
\usepackage{comment}
\usepackage{amsfonts,amssymb,mathrsfs,amscd}
\usepackage{enumitem}
\usepackage{etoolbox}

\newtheorem{proposition}{Proposition}[section]
\newtheorem{theorem}[proposition]{Theorem}
\newtheorem{corollary}[proposition]{Corollary}
\newtheorem{lemma}[proposition]{Lemma}

\theoremstyle{definition}
\newtheorem{definition}[proposition]{Definition}

\newtheorem{remark}[proposition]{Remark}

\numberwithin{equation}{section}

\newcommand{\rot}{\mathop\mathrm{curl}}

\renewcommand{\div}{\mathop\mathrm{div}}
\newcommand{\Tr}{\mathop\mathrm{Tr}}
\def\R{\mathbb R}

\def\({\left(}
\def\){\right)}
\def\divv{\operatorname{div}}

\def\Dx{\Delta_x}
\def\Nx{\nabla_x}
\def\Dt{\partial_t}
\def\Tr{\operatorname{Tr}}
\def\Bbb{\mathbb}
\def\Cal{\mathcal}
\begin{document}
 \title[Damped 3D Euler--Bardina equations]{ Sharp upper and lower bounds of the attractor dimension for   3D damped Euler--Bardina equations}
\author[A. Ilyin, A. Kostianko, and S. Zelik] {Alexei Ilyin${}^1$, Anna Kostianko${}^{3,4}$,
and Sergey Zelik${}^{1,2,3}$}



\thanks{This work was supported by Moscow Center for Fundamental and Applied Mathematics,
Agreement with the Ministry of Science and Higher Education of the Russian Federation,
No. 075-15-2019-1623 and by the Russian Science Foundation grant No.19-71-30004 (sections 2-4).
The second  author was  partially supported  by the Leverhulme grant No. RPG-2021-072 (United Kingdom).}

\email{ilyin@keldysh.ru}
\email{aNNa.kostianko@surrey.ac.uk}
\email{s.zelik@surrey.ac.uk}
\address{${}^1$ Keldysh Institute of Applied Mathematics, Moscow, Russia}
\address{${}^2$ University of Surrey, Department of Mathematics, Guildford, GU2 7XH, United Kingdom.}
\address{${}^3$ \phantom{e}School of Mathematics and Statistics, Lanzhou University, Lanzhou\\ 730000,
P.R. China}
\address{${}^4$  Imperial College, London SW7 2AZ, United Kingdom.}

\begin{abstract}
The dependence of the fractal dimension of global attractors for
the damped 3D Euler--Bardina equations on the regularization
parameter $\alpha>0$ and Ekman damping coefficient $\gamma>0$ is
studied. We present  explicit upper bounds for this dimension for
the case of the whole space, periodic boundary conditions, and the
case of bounded domain with Dirichlet boundary conditions. The
sharpness of these estimates when  $\alpha\to0$ and $\gamma\to0$
(which corresponds in the limit to the classical Euler equations)
is demonstrated on the  3D Kolmogorov flows on a torus.
\end{abstract}

\maketitle
\small\tableofcontents

\setcounter{equation}{0}

\section{Introduction}\label{sec0}

Being the central mathematical model in hydrodynamics, the Navier-Stokes and Euler equations permanently remain in the focus of both the analysis of PDEs and the theory of infinite dimensional dynamical systems and their attractors, see \cite{B-V,Ch-V-book,Fef06,FMRT,F95,Lad,L34,tao,T,T95} and the references therein for more details. Most studied is the 2D case where reasonable results on the global well-posedness and regularity of solutions as well as the results on the  existence of global attractors and their dimension are available. However, the global well-posedness  in
 the 3D case remains a mystery and even listed by the Clay institute of mathematics as one of the Millennium problems. This mystery inspires a comprehensive study of various modifications/regularizations of the initial Navier-Stokes/Euler equations (such as Leray-$\alpha$ model, hyperviscous Navier-Stokes equations, regularizations via $p$-Laplacian, etc.), many of which have a strong physical background and are of independent interest, see e.g. \cite{Camassa,HLT10,L34,Lopes,OT07} and the references therein.
\par
In the present paper we shall be dealing
 with the following regularized damped Euler system:
\begin{equation}\label{DEalpha}
\left\{
  \begin{array}{ll}
    \partial_t u+(\bar u,\nabla_x)\bar u+\gamma u+\nabla_x p=g,\  \  \\
    \operatorname{div} \bar u=0,\quad u(0)=u_0.
  \end{array}
\right.
\end{equation}
with  forcing $g$ and  Ekman damping term $\gamma u$,
$\gamma>0$. The damping term $\gamma  u$ makes the system dissipative and
is important in  various geophysical models~\cite{Ped}. Here and below  $\bar u$ is a smoothed (filtered)
vector field related with the initial velocity field $u$ as the
solution of the Stokes problem
\begin{equation}\label{0.bar}
u=\bar u-\alpha\Dx \bar u+\Nx q,\ \ \divv\bar u=0,
\end{equation}
where $\alpha>0$ is a given small parameter. In other words,
$$
\bar u=(1-\alpha A)^{-1}u,
$$
where $A:=\Pi\Dx$ is the Stokes operator and $\Pi$ is the Helmholtz--Leray projection to divergent free vector fields in the corresponding domain.
\par
System \eqref{DEalpha}, \eqref{0.bar} (at least in the conservative case $\gamma=0$) is
often referred to as the  simplified Bardina subgrid
scale model of turbulence, see \cite{BFR80,Bardina,Lay} for the derivation
of the model and further discussion, so in this paper
we  shall be calling  \eqref{DEalpha}  the damped Euler--Bardina equations.
We also mention that rewriting \eqref{DEalpha} in terms of the variable $\bar u$ gives
\begin{equation}\label{0.KV}
\Dt\bar u-\alpha\Dt\Dx\bar u+(\bar u,\Nx)\bar u+\gamma\bar u+\Nx p=\alpha\gamma\Dx \bar u+g
\end{equation}
which is a damped version of the so-called Navier--Stokes--Voight equations arising
in the theory of viscoelastic fluids, see \cite{Titi-Varga,Osk} for the details.
\par
Our main interest in the present paper is to study the dimension of
global attractors for system \eqref{DEalpha} in 2D and 3D paying main
attention to the most complicated 3D case. Note that, unlike the classical
Euler equations, Bardina-Euler equations can be interpreted as an ODE with bounded nonlineariry
in the
proper Hilbert space, so no problems with well-posedness arise, see \cite{Bardina}
and also section \S\ref{sec2} below, so the main aim of our study is
to get as sharp as possible bounds for the corresponding global attractors.
Each case $d=2$ and $d=3$ in
turn is studied in three different settings as  far as the boundary
conditions are concerned. More precisely, the system is studied
\begin{enumerate}
  \item on the torus $\Omega=\mathbb{T}^d=[0,2\pi]^d$. In this case the
      standard zero mean condition is imposed on  $u$, $\bar u$ and $g$;
  \item in the whole space $\Omega=\mathbb{R}^d$;
  \item in a bounded domain $\Omega\subset\mathbb{R}^d$  with Dirichlet boundary conditions
  for $\bar u$.
\end{enumerate}
 We denote by $W^{s,p}(\Omega)$ the standard Sobolev  space of distributions whose
 derivatives up to order $s$ belong to the Lebesgue space $L^p(\Omega)$. In the
 Hilbert case $p=2$ we will write $H^s(\Omega)$ instead of $W^{s,2}(\Omega)$.
 In order to work with velocity vector fields, we denote by ${\bf H}^s={\bf H}^s(\Omega)$
 the subspace of $[H^s(\Omega)]^d$ consisting of   divergence free vector fields.
 In the case of $\Omega\subset\R^d$ we assume in addition that vector fields from
 ${\bf H}^s$ satisfy Dirichlet boundary conditions and in the case of periodic
 boundary conditions $\Omega=\Bbb T^d$ we assume that these vector fields have zero mean.
 We also recall that equation \eqref{DEalpha} possesses the standard energy identity
 $$
 \frac12\frac d{dt}\(\|\bar u\|^2_{L^2(\Omega)}+\alpha\|\Nx \bar u\|^2_{L^2(\Omega)}\)+
 \gamma\(\|\bar u\|^2_{L^2(\Omega)}+\alpha\|\Nx \bar u\|^2_{L^2(\Omega)}\)=
(g,\bar u),
$$
where $(u,v)$ is the standard inner product in $[L^2(\Omega)]^d$.
For this reason it is natural to consider problem \eqref{DEalpha}
in the phase space ${\bf H}^1$  with  norm
$$
\|\bar u\|^2_{\alpha}:=\|\bar u\|^2_{L^2}+\alpha\|\Nx \bar u\|^2_{L^2}.
$$
Our first main result is the following theorem which gives an explicit upper
bound for the fractal dimension of the attractor in the 3D case.

\begin{theorem}Let $d=3$, let $\Omega$ be as described above, and
let  $g\in [L^2(\Omega)]^3$ (in the periodic case we assume also
that $g$ has zero mean). Then the solution semigroup $S(t)$ associated
with equation \eqref{DEalpha}  possesses
a global attractor $\mathscr A \Subset{\bf H}^1$ with finite fractal dimension
satisfying the following inequality:
\begin{equation}\label{0.est1}
\dim_F\mathscr A\le\frac{1}{12\pi}\frac{\| g\|_{L^2}^2}{\alpha^{5/2}\gamma^4}\,.
\end{equation}
The analogue of this estimate for the  2D case reads
\begin{equation}\label{0.est2}
\dim_F\mathscr A\le \frac{1}{16\pi}\frac{\| g\|_{L^2}^2}{\alpha^{2}\gamma^4}
\end{equation}
with the following improvement for the case when $\Omega=\Bbb T^2$ or $\Omega=\R^2$:
\begin{equation}\label{0.est3}
\dim_F\mathscr A\le \frac{1}{8\pi}\frac{\| \rot g\|_{L^2}^2}{\alpha\gamma^4}
\end{equation}
due to estimates related with the vorticity equation.
\end{theorem}
Since the general case $\gamma>0$ is reduced to the particular one with  $\gamma=1$ by
scaling $t\to\gamma^{-1}t$, $u\to\gamma^{-2}u$, $g\to\gamma^{-2}g$,
the most interesting in estimates \eqref{0.est1}, \eqref{0.est2}
and \eqref{0.est3} is the dependence of the RHS on $\alpha$. For
the viscous case of equations \eqref{DEalpha}
$$
\Dt u+(\bar u,\Nx)\bar u+\Nx q=\nu\Dx u+g
$$
the following estimate is proved in \cite{Bardina}:
$$
\dim_F\mathscr A\le C\frac{\|g\|_{L^2}^{6/5}}{\nu^{12/5}\alpha^{18/5}}
$$
for the case $\Omega=\Bbb T^3$. We see that even in the case $\nu=1$ this
estimate gives essentially worse dependence on $\alpha$ than our estimate \eqref{0.est1}.
The upper bounds for 3D Navier-Stokes-Voight equation obtained
in~\cite{Titi-Varga} give even worse dependence on the parameter $\alpha$ (like $\alpha^{-6}$).
 Estimates \eqref{0.est2} and \eqref{0.est3} have been proved
for $\Omega=\Bbb T^2$ in a recent paper~\cite{IZLap70}. The sharpness of these
estimates in the limit as $\alpha\to0$ was also established
there for the case of the 2D Kolmogorov flows. However, to the best of our knowledge, no
lower bounds for the
dimension of the attractor of the Euler--Bardina equations in 3D are available
in the literature.
\par
Our second main result covers this gap. Namely, we consider the 3D Kolmogorov
 flows on the  torus $\Omega=\Bbb T^3$ for equations \eqref{DEalpha} generated by the family
 of the right-hand sides parameterized by an integer parameter $s\in\mathbb{N}$:
\begin{equation}\label{0.kol}
g=g_s=\begin{cases} g_1=\gamma^2\lambda(s)\sin(sx_3),\\
g_2=0,\\ g_3=0,
            \end{cases}
\end{equation}
where $s\sim\alpha^{-1/2}$ and $\lambda(s)$ is a specially chosen amplitude, see \S\ref{sec4}.
Then, performing an accurate instability analysis for the linearization of equation \eqref{DEalpha}
on the corresponding Kolmogorov flow (in the spirit of \cite{Liu2}, see also \cite{IT1,IMT,Liu}),
we get the following result.

\begin{theorem} Let $\Omega=\Bbb T^3$ and let $\gamma>0$, and $\alpha>0$. Then in the limit
$\alpha\to0$ the integer parameter $s$ and  the amplitude $\lambda(s)$
can be chosen so that the corresponding forcing  $g=g_s$ of the form  \eqref{0.kol} produces
the global attractor $\mathscr A=\mathscr A_s$, whose
dimension satisfies the following lower bound:
\begin{equation}\label{0.est4}
\dim_F\mathscr A\ge c\frac{\| g\|_{L^2}^2}{\alpha^{5/2}\gamma^4},
\end{equation}
where $c>0$ is an absolute effectively computable constant.
\end{theorem}
Estimate \eqref{0.est4} shows that our upper bound \eqref{0.est1} is optimal. Again,
 to the best of our knowledge, this is the first optimal two-sided estimate for the attractor
 dimension in a 3D hydrodynamical problem.

In this connection we recall the celebrated upper bound in \cite{CFT} for the
attractor dimension of the
classical Navier--Stokes system on the 2D torus,
which is still  logarithmically larger than the corresponding lower bound in \cite{Liu}.
On the other hand, adding to the system an arbitrary fixed damping
makes it possible to obtain the estimate for the attractor dimension that is optimal
in the vanishing viscosity limit~\cite{IMT}.

\par
 We finally observe that the obtained lower estimates  
for the attractor
dimension grow as $\alpha\to0$ in both 2D and 3D cases (and even are optimal for the case of tori), so 
one may expect that the limit attractor $\Cal A_0$ (which corresponds to the case of non-modified damped Euler equation) is infinite dimensional. Indeed, the existence of the attractor $\Cal A_0$ in the proper phase space is well-known in 2D at least if $g\in W^{1,\infty}$, see \cite{CIZ} and references therein and we expect that some weaker version of the limit attractor $\Cal A_0$ can be also constructed in  3D using the trajectory approach, see \cite{Ch-V-book}, and the concept of dissipative solutions for 3D Euler introduced by P. Lions, see \cite{PLio}.
However, the situation with the dimension is much more delicate since the obtained lower bounds for the instability index on  Kolmogorov's flows are optimal for intermediate values of $\alpha$ only and do not provide any reasonable bounds for the limit case $\alpha=0$. Thus,
the question of finite or infinite-dimensionality of the limit attractor remains completely open even in the 2D case.

\par
The paper is organized as follows. The key estimates for the solutions of problem
\eqref{DEalpha} are derived in  \S\ref{sec2}. Global well-posedness and dissipativity
are also discussed there. The existence of a global attractor $\mathscr A$ is verified in
 \S\ref{sec3}. To make the proof independent of the choice of a (bounded or unbounded) domain
$\Omega$, we use the so called energy method for establishing the asymptotic compactness
of the associated semigroup.
\par

The upper bounds for its dimension are obtained in  \S\ref{sec4}
via the volume contraction method~\cite{B-V,CF85,T}. The essential role in getting
optimal bounds for the global Lyapunov exponents is played by the collective
Sobolev inequalities for $H^1$-orthonormal families
 proved in Appendix~\ref{sec5} based on the ideas of \cite{LiebJFA}.
 Their role is somewhat similar to the role of the Lieb--Thirring inequalities
 \cite{Lieb,LT} in the dimension estimates of the attractors
 of the classical Navier--Stokes equations \cite{B-V,T}.
The corresponding  inequality in the 2D case has also been  used in \cite{IZLap70}.
 Finally, the sharp lower bounds of the dimension for the case
 $\Omega=\Bbb T^3$ are obtained in  \S\ref{sec4} by adapting/extending the ideas
 of \cite{IZLap70,Liu2} to the 3D case.

\setcounter{equation}{1}
\section{A priori estimates, well-posedness and dissipativity }\label{sec2}

We start with the standard energy estimate, which looks the same in the 2D and  3D cases as
well as for the three types of boundary conditions.

\begin{proposition}\label{Prop:2.1} Let $u$ be a sufficiently regular solution of equation \eqref{DEalpha}.
Then the following dissipative energy estimate holds:
\begin{equation}\label{2.en}
\|\bar u(t)\|^2_\alpha\le \|\bar u(0)\|^2_\alpha e^{-\gamma t}+\frac1{\gamma^2}\|g\|^2_{L^2},
\end{equation}
where
\begin{equation}\label{2.norm}
\|\bar u\|^2_\alpha:=\|\bar u\|^2_{L^2}+\alpha\|\Nx\bar u\|^2_{L^2}.
\end{equation}
\end{proposition}
\begin{proof} Indeed, multiplying equation \eqref{DEalpha} by $\bar u$, integrating over $\Omega$ and using the relation between $u$ and $\bar u$ as well as the standard fact that the inertial term vanishes after the integration, we arrive at
\begin{multline}\label{2.en-est}
\frac d{dt}\left(\|\bar u\|^2_{L^2}+\alpha\|\nabla_x\bar u\|^2_{L^2}\right)+
2\gamma\left(\|\bar u\|^2_{L^2}+\alpha\|\nabla_x\bar u\|^2_{L^2}\right)=2(g,\bar u)\le\\\le
2\|g\|_{L^2}\|\bar u\|_{L^2}\le
\gamma\|\bar u\|^2_{L^2}+\frac1\gamma\|g\|_{L^2}^2.
\end{multline}
Applying the Gronwall inequality, we get the desired estimate \eqref{2.en}
and complete the proof.
\end{proof}
The next  corollary is crucial for our upper bounds for the attractor dimension.

\begin{corollary}\label{Cor2.en-av} Let $u$ be a sufficiently smooth solution of problem \eqref{DEalpha}. Then the following estimate holds:
\begin{equation}\label{2.est-int}
\limsup_{t\to\infty}\frac1t\int_0^t\|\Nx u(s)\|_{L^2}\,ds\le \frac1{\gamma\sqrt{2\alpha}}\|g\|_{L^2}.
\end{equation}
\end{corollary}
\begin{proof} Indeed, integrating estimate \eqref{2.en-est} over $t$, taking the limit $t\to\infty$ and using the fact that $\|u(t)\|_\alpha^2$ remains bounded (due to estimate \eqref{2.en}, we arrive at
$$
\limsup_{t\to\infty}\frac1t\int_0^t\|\Nx u(s)\|^2_{L^2}\,ds\le \frac1{2\alpha\gamma^2}\|g\|^2_{L^2}.
$$
Using after that the H\"older inequality
$$
\frac1t\int_0^t\|\Nx u(s)\|_{L^2}\,ds\le
 \(\frac1t\int_0^t\|\Nx u(s)\|^2_{L^2}\,dx\)^{1/2},
$$
we get the desired result and finish the proof of the corollary.
\end{proof}
We now turn to the two dimensional case without boundary. In this case, more accurate estimates are available due to the possibility to use the vorticity equation. Indeed,
applying $\operatorname{curl}$ to~\eqref{DEalpha} and
setting  $\omega=\operatorname{curl}u$, we obtain the vorticity equation
for $\omega$:
\begin{equation}\label{vort}
\partial_t\omega+(\bar u,\nabla_x)\bar\omega+\gamma\omega=\operatorname{curl} g,\ \omega=(1-\alpha\Delta_x)\bar\omega.
\end{equation}
 The estimates for the solution
on the torus $\mathbb{T}^2$ were derived in~\cite{IZLap70}. Although for
$\mathbb{R}^2$ they are formally the same, we reproduce them for the sake of completeness.

\begin{proposition} Let $u$ be a sufficiently smooth solution of \eqref{DEalpha},
where $\Omega=\mathbb{T}^2$ or $\mathbb{R}^2$ and let $\omega:=\rot u$ and $\bar\omega:=\rot \bar u$. Then, the following dissipative estimate holds:
\begin{equation}\label{2.en-vor}
\|\bar\omega(t)\|^2_{\alpha}\le \|\bar\omega(0)\|^2_{\alpha}e^{-\gamma t}+\frac1{\gamma^2}\|\rot g\|^2_{L^2}.
\end{equation}
\end{proposition}
\begin{proof}
Taking the scalar product of equation \eqref{vort} with $\bar\omega$, we see that the
nonlinear term vanishes and using that
\begin{equation}\label{omom}
(\omega,\bar\omega)=\|\bar \omega\|^2_{L^2}+\alpha\|\nabla_x\bar\omega\|^2_{L^2},
\end{equation}
 we obtain
\begin{multline}
\frac12\frac d{dt}\left(\|\bar \omega\|^2_{L^2}+\alpha\|\nabla_x\bar\omega\|^2_{L^2}\right)+
\gamma\left(\|\bar\omega\|^2_{L^2}+\alpha\|\nabla_x\bar\omega\|^2_{L^2}\right)=
(\rot g,\bar\omega)\le\\\le\|\rot g\|_{L^2}\|\bar\omega\|_{L^2}\le
\frac1{2\gamma}\|\rot g\|^2_{L^2}+\frac\gamma2\|\bar\omega\|^2_{L^2}.
\end{multline}
This gives the desired estimate \eqref{2.en-vor} by the  Gronwall inequality and
finishes the proof of the proposition.
\end{proof}
Analogously to Corollary \ref{Cor2.en-av}, we get the following estimate.

\begin{corollary}\label{Cor2.en-av2} Let $\Omega=\mathbb{T}^2$ or $\mathbb{R}^2$ and let
$u$ be a sufficiently smooth solution of problem \eqref{DEalpha}. Then the following estimate holds:
\begin{equation}
\limsup_{t\to\infty}\frac1t\int_0^1\|\Nx\bar u(s)\|_{L^2}\,ds\le \frac1\gamma\min\left\{\|\rot g\|_{L^2},\frac{\|g\|_{L^2}}{\sqrt{2\alpha}}\right\}.
\end{equation}
\end{corollary}
Indeed, the second inequality was already proved in Corollary
\ref{Cor2.en-av} and the first one is an immediate corollary of
\eqref{2.en-vor} and the fact that
$$
\|\nabla\bar u\|_{L^2}=\|\bar \omega\|_{L^2}.
$$

Let us conclude this section by discussing the well-posedness of
problem \eqref{DEalpha} and justification of the estimates obtained
above. We will consider below only the  3D case (the 2D case is
analogous and  even slightly simpler).
\par
We also note from the very beginning that equation \eqref{DEalpha}
can be rewritten in the form of   an ODE in a Hilbert space with
bounded nonlineariry. Indeed, applying the Helmholtz--Leray
projection $\Pi$ to both sides of \eqref{DEalpha} together with the
operator
$$
A_\alpha:=(1-\alpha A)^{-1},
$$
where $A=\Pi\Dx$ is the Stokes operator in $\Omega$, we arrive at
   \begin{equation}\label{1.ODE}
   \Dt\bar u+\gamma \bar u+B(\bar u,\bar u)=A_\alpha \Pi g, \ \bar u\big|_{t=0}=\bar u_0,
   \end{equation}
   where $B(\bar u,\bar v):=  A_{\alpha}\Pi\((\bar u,\Nx)\bar v\)$.
  \par
It is natural to  consider this system in the phase space
$\bar u\in{\bf H}^1(\Omega)$  with norm \eqref{2.norm}.
Then the nonlinear operator $B$ is bounded from $\mathbf{H}^1$
to $\mathbf{H}^{3/2}$:
\begin{equation}\label{1.B}
\|B(\bar u,\bar v)\|_{\mathbf{H}^{3/2}}\le
C_\alpha\|\bar u\|_{\alpha}\|\bar v\|_{\alpha},
\end{equation}
where $C_\alpha$ depends only on $\alpha$.
Indeed, if $\bar u,\bar v\in {\bf H}^1$, then by the Sobolev embedding theorem
 $\bar u,\bar v\in L^6(\Omega)$ and
 $(\bar u,\Nx)\bar v\in L^{3/2}(\Omega)$ by H\"older's inequality. Together with the
 $(L^{3/2}\to W^{2,3/2})$-boundedness of the operator $(1-\alpha A)^{-1}$,
 we get that $B(\bar u,\bar v)\in W^{2,3/2}(\Omega)$.
 Finally, the Sobolev embedding $W^{2,3/2}\subset H^{3/2}$ proves
 estimate \eqref{1.B}.

Thus, $B(\bar u,\bar u)$ is a regularizing
operator in ${\bf H}^1$ and equation \eqref{1.ODE} is an ODE in ${\bf H}^1$
with bounded nonlineariry. Therefore the local existence
and uniqueness of a solution as well as (an infinite) differentiability of the
corresponding local solution semigroup are straightforward corollaries of the
Banach contraction principle  or the
implicit function theorem, see e.g. \cite{henry} for the details.
Thus, to get the global well-posedness and dissipativity
 we only need to verify the proper a priori estimate.
Since this  has already been done in Proposition~\ref{Prop:2.1},
we have proved the following theorem.

    \begin{theorem}\label{Th1.exist} Let $\bar u_0\in {\bf H}^1(\Omega)$, $g\in [L^2(\Omega)]^d$
    (in the case of periodic BC we also assume that $g$ has zero mean).
     Then there exists a unique global solution $\bar u\in C([0,\infty),{\bf H}^1)$
    of problem \eqref{1.ODE} (which is simultaneously the unique solution of
    \eqref{DEalpha}). Moreover, the function
    $$
    t\to \|\bar u(t)\|_{L^2}^2+\alpha\|\Nx \bar u(t)\|^2_{L^2}
    $$
    is absolutely continuous
     and the following   energy identity holds:
    \begin{multline}\label{1.energy}
     \frac12\frac d{dt}\(\|\bar u(t)\|_{L^2}^2+\alpha\|\Nx \bar u(t)\|^2_{L^2}\)+\\+
     \gamma\(\|\bar u(t)\|_{L^2}^2+\alpha\|\Nx \bar u(t)\|^2_{L^2}\)=(g,\bar u).
    \end{multline}
    In particular, the dissipative estimate \eqref{2.en} holds for
    any solution $u$ of class $u\in C([0,\infty),{\bf H}^1)$.
    \end{theorem}

\begin{corollary}\label{Cor1.sem} Let the assumptions of Theorem \ref{Th1.exist} holds. Then equation \eqref{1.ODE}
generates a dissipative solution semigroup
\begin{equation}\label{1.s}
S(t)\bar u_0:= \bar u(t),\ \ t\ge0
\end{equation}
in the phase space ${\bf H}^1(\Omega)$. Moreover, $S(t)$ is $C^\infty$-differentiable for every fixed~$t$.
\end{corollary}
Indeed, the existence of the semigroup is an immediate corollary of the well-posedness proved in
the theorem and the differentiability follows from the ODE structure
of \eqref{1.ODE} and the fact that
the map $\bar u\to B(\bar u,\bar u)$ is $C^\infty$-smooth as
a map from ${\bf H}^1$ to ${\bf H}^1$.

\section{Asymptotic compactness and attractors}\label{s2}
In this section we construct a global attractor for the solution semigroup $S(t)$
generated  by problem \eqref{DEalpha}. We start with recalling the definition
of a weak and  strong   global attractor, see \cite{B-V,Ch-V-book} for more details.
We will mainly consider below the most complicated case $\Omega=\R^3$
since in the case of a bounded domain the asymptotic compactness
is an immediate corollary of the fact that $B(\bar u,\bar u)\in \mathbf{H}^{3/2}$
if $u\in {\mathbf H}^1$, see Remark \ref{Rem1.useless}.

\begin{definition} A set $\mathscr A_w\subset {\bf H}^1$ is a weak global attractor of the
semigroup $S(t)$ if
\par
1) $\mathscr A_w$ is a compact set in ${\bf H}^1$  with weak topology;
\par
2) $\mathscr A_w$ is strictly invariant, i.e., $S(t)\mathscr A_w=\mathscr A_w$;
\par
3) $\mathscr A_w$ attracts the images of all bounded sets in the weak topology of ${\bf H}^1$, i.e. for every bounded
set $B\subset {\bf H}^1$ and every neighbourhood $\mathcal O(\mathscr A_w)$ of the
attractor in the weak topology,
 there exists $T=T(\Cal O, B)$ such that
 $$
 S(t)B\subset\Cal O(\mathscr A_w) \ \text{for all} \ t\ge T.
 $$
Analogously, $\mathscr A_s$ is a strong attractor if it is compact in the strong
topology of ${\bf H}^1$, is strictly invariant and attracts the
images of bounded sets in the strong topology as well. Obviously
$$
\mathscr A_w=\mathscr A_s
$$
if both attractors exist.
\end{definition}

We will use the following criterion for verifying the existence of
an attractor, see \cite{B-V,T} for the details.

\begin{proposition}\label{Prop1.crit} Let the operators  operators $S(t)$ be
continuous in the weak topology for every fixed $t$ and let the semigroup $S(t)$
possess a bounded absorbing set $\Cal B$. The
latter means that for every bounded $B\subset {\bf H}^1$
there exists $T=T(B)$ such that
 $$
 S(t)B\subset \Cal B\ \ \text{for all}\ \ t\ge T.
 $$
 Then there exists a weak global attractor $\Cal A_w$ of the semigroup $S(t)$ which is
  generated by all complete (defined for all $t\in\R$) bounded solutions of problem \eqref{1.ODE}:
  \begin{equation}
\Cal A_w=\Cal K\big|_{t=0},
  \end{equation}
  where $\Cal K:=\{\bar u\in C_b(\R, {\bf H}^1),\ \ \bar u  \text{ solves } \eqref{1.ODE}\}$.
\par
Let, in addition, $S(t)$ be asymptotically compact on $\Cal B$. The latter means that
 for every sequence
 $\bar u_0^n\in \Cal B$ and every sequence $t_n\to\infty$, the sequence
 $$
 \{S(t_n)\bar u_0^n\}_{n=1}^\infty
 $$
 is precompact in the strong topology of ${\bf H}^1$. Then $\mathscr A_w$ is also a strong global attractor
 for the semigroup $S(t)$.
\end{proposition}
We start with verifying the existence of a weak attractor.

\begin{proposition}\label{Prop1.attr} Let the assumptions of Theorem \ref{Th1.exist} hold. Then the solution semigroup
 $S(t)$ generated by equation \eqref{DEalpha} possesses a weak global attractor $\Cal A_w$ in
  the phase space ${\bf H}^1$.
\end{proposition}
\begin{proof} The existence of a bounded absorbing set $\Cal B$ is an immediate corollary of the
 dissipative estimate \eqref{2.en}. We may take
 $$
 \Cal B:=\{\bar u\in {\bf H}^1,\ \ \|\bar u\|^2_{L^2}+\alpha\|\Nx\bar u\|^2_{L^2}\le
 \frac2{\gamma^2}\|g\|_{L^2}^2\}.
$$
Thus, we only need to check the weak continuity. Let $\bar u_0^n\in\Cal B$
be a sequence of  the initial data weakly converging to $\bar u_0$:
$\bar u_0^n\rightharpoondown \bar u_0$ in ${\bf H}^1$.
Denote by  $\bar u^n(t):=S(t)\bar u_0^n$
 the corresponding solutions. We need to check that for every fixed $T$,
 $\bar u^n(T)\rightharpoondown \bar u(T)$ in ${\bf H}^1$, where $\bar u(t):=S(T)\bar u_0$.
 \par

To see this  we recall that $\bar u^n$ is bounded uniformly with respect to
$n$ in $L^\infty(0,T;{\bf H}^1)$ due   to estimate \eqref{2.en}. Moreover,
from equation \eqref{1.ODE} we see also that $\Dt\bar u^n$ is
uniformly bounded in the same space. Thus, passing to a subsequence, if necessary,
 we may assume that
$\bar u^n(t)\rightharpoondown v(t)$ for every $t\in[0,T]$ and
$\Dt \bar u_n\rightharpoondown\Dt v$
in $L^2(0,T;{\bf H}^1)$ for some function $v(t)$ such that $v,\Dt v\in L^\infty(0,T;{\bf H}^1)$. So,
   it remains to verify that $v(t)=S(t)\bar u_0$ by passing to the limit in equations \eqref{1.ODE}
   for functions $\bar u^n$.
   \par
   This passing to the limit is obvious for linear terms, so we only need to prove the convergence
   of the nonlinear term $B(\bar u_n,\bar u_n)$. In turn, this is the same as to prove that, in the sense of distributions,
   $$
   (\bar u^n,\Nx)\bar u^n=\divv(\bar u^n\otimes \bar u^n)\rightharpoondown
    \divv(v\otimes v)=(v,\Nx)v.
   $$
The last statement will be proved if we check that
\begin{equation}\label{1.triv}
\bar u^n\otimes\bar u^n\rightharpoondown v\otimes v \ \ \text{in}\ \ L^2((0,T)\times\Omega).
\end{equation}
To verify \eqref{1.triv}, we recall that the sequence $\bar u^n\otimes\bar u^n$ is uniformly bounded in
$L^2$ due to  dissipative estimate \eqref{2.en} and the embedding
$H^1((0,T)\times\Omega)\subset L^4$.
Moreover, since  the embedding
$H^1((0,T)\times\R^3)\subset L^2((0,T);L^2_{loc}(\Omega))$ is compact,
we have the strong convergence $\bar u^n\to v$ in $L^2((0,T);L^2_{loc}(\Omega))$ and, therefore, the
 convergence $\bar u^n\to v$ almost everywhere. Since the sequence
 $\bar u^n\otimes\bar u^n$ is uniformly bounded in $L^2((0,T)\times\Omega)$,
 we may assume without loss of generality that
 it is weakly convergent to some  $\psi\in L^2((0,T)\times\Omega)$.
Along with the established convergence almost everywhere this implies
that $\psi=v\otimes v$, see e.g. \cite{lions}, and proves \eqref{1.triv}.
\par
Thus, we have proved that $v$ solves the equation \eqref{1.ODE} and by the
uniqueness $v(t)=\bar u(t)$. This finishes the proof of weak continuity of the
operators $S(t)$ and   the existence of a weak global attractor now follows
from Proposition \ref{Prop1.crit}. The theorem is proved.
\end{proof}

We are now ready to verify the existence of a strong global attractor.

\begin{proposition}\label{Prop1.attr-s} Let the assumptions of Theorem \ref{Th1.exist} hold. Then
 the solution semigroup $S(t)$ generated by equation \eqref{1.ODE} possesses a strong global attractor
  $\mathscr A=\mathscr  A_s$ in the phase space ${\bf H}^1$.
\end{proposition}

\begin{proof} According to Proposition \ref{Prop1.crit}, we only need to verify the asymptotic
compactness of $S(t)$ on $\Cal B$. We will use the so-called energy method for this
purpose, see \cite{ball,rosa} for more details.
\par
Let $\{\bar u^n_0\}\subset\Cal B$, let $t_n\to\infty$ be arbitrary and let
$\bar u^n(t):=S(t_n)\bar u^n_0$. Define also $\bar v^n(t):=\bar u^n(t+t_n)$. Then these
functions are defined on the time intervals $t\in[-t_n,\infty)$ and,
due to the existence
 of a weak global attractor, without loss of generality, we may assume that
  $\bar v(t)\rightharpoondown \bar u(t)$ in ${\bf H}^1$ for all $t\in\R$ to some
   complete trajectory $\bar u\in\Cal K$. In particular,
\begin{equation}\label{1.weak}
\bar v^n(0)=S(t_n)\bar u_0^n\rightharpoondown\bar u(0)
\end{equation}
and we only need to check that this convergence is strong.
\par
It is convenient to use the equivalent norm~\eqref{2.norm}  in the space ${\bf H}^1$.
Then, the strong convergence in \eqref{1.weak} will be proved if we verify that
\begin{equation}\label{norm-to-norm}
\|\bar v^n(0)\|^2_\alpha\to\|\bar u(0)\|^2_\alpha.
\end{equation}
To see this we integrate the energy identity \eqref{1.energy} for $\bar v^n(t)$
 in time and get
\begin{equation}\label{1.en-finite}
\|\bar v^n(0)\|^2_\alpha=\|\bar u_0^n\|^2_\alpha e^{-2\gamma t_n}+
\int_{-t_n}^0e^{2\gamma s}(g,\bar v^n(s))\,ds.
\end{equation}
Passing to the limit $n\to\infty$ in this relation and using the weak convergence of
$\bar v^n$ to  $\bar u$ and uniform boundedness of $\bar v^n$ and the initial data
$\bar u_0^n$,  we conclude that
 \begin{equation}\label{1.en-lim}
\lim_{n\to\infty}\|\bar v^n(0)\|^2_\alpha=\int_{-\infty}^0e^{2\gamma s}(g,\bar u(s))\,ds.
 \end{equation}
On the other hand, integrating the energy identity for the limit solution $\bar u$
in time, we arrive at
 \begin{equation}\label{1.en-lim1}
\|\bar u(0)\|^2_\alpha=\int_{-\infty}^0e^{2\gamma s}(g,\bar u(s))\,ds.
 \end{equation}
Equalities \eqref{1.en-lim} and \eqref{1.en-lim1} imply
\eqref{norm-to-norm}, therefore the convergence in \eqref{1.weak}
 is actually strong. Thus, the desired asymptotic compactness is proved and the proposition is also proved.
\end{proof}
\begin{remark}\label{Rem1.useless}
{\rm Since the operator $B(\bar u,\bar u)$ is
regularizing, one can easily
increase the regularity of the global attractor $\mathscr A$ using the  decomposition
of the  semigroup into the decaying linear part and the regularizing nonlinear part
(see \cite{Aro}):
 $$
 S(t):=L(t)+K(t),
 $$
 where $v(t)=L(t)\bar u_0 $ solves
 $$
 \Dt v+\gamma v=0,\ \ v\big|_{t=0}=\bar u_0
 $$
 and $w(t):=K(t)\bar u_0$ satisfies
 $$
\Dt w+\gamma w+B(\bar u,\bar u)= A_\alpha \Pi g,\ \ w\big|_{t=0}=0.
 $$
Combining this decomposition with bootstrapping arguments, we may check that the
regularity of the attractor $\mathscr A$ is restricted by the regularity of $g$ only and
it will be $C^\infty$-smooth if $g\in H^\infty(\R^3)$. Moreover, using the proper
weighted estimates, see \cite{MirZel}, we may
get the estimates on the rate of decay for solutions belonging to the attractor as
$|x|\to\infty$ in terms of the decay rate of $g$ which clarify the reason why
$\mathscr A$ is compact. However,
all these estimates do not seem  very helpful for estimation of the attractor
dimension     (since they grow rapidly with respect to $\gamma,\alpha\to0$) and
therefore  we will not go into  more      details here.}
\end{remark}

 \setcounter{equation}{0}

\section{Upper bounds for the fractal dimension}\label{sec3}
In this section we derive upper bounds for the fractal dimension of
the attractor~$\mathscr A$. As usual for the Navier--Stokes type
equations, these bounds will be obtained by means of the  volume
contraction method, see \cite{B-V,CF85,T} and the references
therein. On the analytical side, the Lieb--Thirring inequalities
for $L^2$-orthonormal families \cite{Lieb,LT} which are an
indispensable tool for the dimension estimates of the attractors
for the Navier--Stokes equations are replaced in our case by the
collective Sobolev inequalities for $H^1$-orthonormal families and
are proved in the Appendix \ref{sec5}.

Furthermore, since system \eqref{DEalpha} in the 2D
case   has already been  studied in \cite{IZLap70}
(for the case $\Omega=\Bbb T^2$), we will concentrate here on the 3D case only.

\begin{theorem}\label{Th:est} Suppose that  $\Omega$ is  either the  3D torus
$\Bbb T^3$, or  a bounded domain $\Omega\subset\R^3$
(endowed with Dirichlet BC), or the whole space $\Omega=\R^3$.
Let $g\in [L^2(\Omega)]^d$ (in the case of $\Bbb T^3$ we also assume that
$g$ has zero mean). Then the
global attractor $\mathscr A$ corresponding to the regularized
damped Euler system \eqref{DEalpha} 
has  finite fractal dimension satisfying the following estimate:
\begin{equation}\label{dim-est}
\dim_F\mathscr{A}\le
\frac{1}{12\pi}\frac{\| g\|_{L^2}^2}{\alpha^{5/2}\gamma^4}\,.
\end{equation}
\end{theorem}

\begin{proof} The solution semigroup   $S(t): {\bf H}^1\to
{\bf H}^1$ is smooth with respect to the initial data (see Corollary \ref{Cor1.sem}),
 so we only need to estimate the $n$-traces for the linearization of equation
 \eqref{1.ODE} over trajectories on the attractor. This
 linearization of  \eqref{DEalpha} reads:
 \begin{equation}
 \begin{cases}
 \Dt \bar\theta=-\gamma\bar\theta-B(\bar u(t),\bar\theta)-B(\bar\theta,\bar u(t)) =:L_{u(t)}\bar \theta,\\
\operatorname{div} \bar\theta=0,\  \bar\theta\big|_{t=0}=\bar\theta_0\in \mathbf H^1(\Omega),
\end{cases}
\end{equation}
where  $B(\bar u,\bar v):=A_{\alpha}\Pi\((\bar u,\Nx)\bar v\)$. In order to
 utilize the  well-known cancelation property
$$
((\bar u,\Nx)\bar \theta,\bar\theta)\equiv 0
$$
for the inertial term in the Navier-Stokes equations, it is natural to endow the space ${\bf H}^1$ with the scalar product
\begin{equation}\label{scal-alpha}
(\bar\theta,\bar\xi)_\alpha=(\bar\theta,\bar\xi)+\alpha(\Nx\bar\theta,\Nx\bar\xi)=((1-\alpha A)\bar\theta,\bar\xi)
\end{equation}
associated with the norm \eqref{2.norm}. Then, using that $\Pi A_\alpha=A_\alpha$ and $\Pi\bar\theta=\bar\theta$, we get the cancelation
\begin{multline*}
(B(\bar u,\bar\theta),\bar\theta)_\alpha=
\left(A_\alpha \Pi(\bar u,\Nx)\bar\theta,(1-\alpha\Delta_x)\bar\theta\right)=\\=
\left(A_\alpha \Pi(\bar u,\Nx)\bar\theta,(1-\alpha\Pi\Delta_x)\bar\theta\right)=
\left(\Pi(\bar u,\Nx)\bar\theta,\bar\theta\right)=
((\bar u,\Nx)\bar\theta,\bar\theta)\equiv0
\end{multline*}
of the most singular term $B(\bar u,\bar\theta)$ and, therefore, only the more
regular term $B(\bar\theta,\bar u)$ will impact the trace estimates.
\par
Following the general strategy, see e.g. \cite{T}, the $n$-dimensional volume
 contraction factors $\omega_n(\mathscr A)$ (=the sums of the first $n$
global Lyapunov exponents) which control the dimension can be estimated from
above by the following numbers:
$$
q(n):=\limsup_{t\to\infty}\sup_{u(t)\in\mathscr A}\sup_{\{\bar\theta_j\}_{j=1}^n}\frac1t\int_0^t
\sum_{j=1}^n(L_{u(\tau)}\bar\theta_j,\bar\theta_j)_\alpha d\tau,
$$
where the first (inner) supremum is taken over all orthonormal families $\{\bar\theta_j\}_{j=1}^n$ with respect to the scalar product $(\cdot,\cdot)_\alpha$ in ${\bf H}^1$:
\begin{equation}\label{alpha-orth}
(\bar\theta_i,\bar\theta_j)_\alpha=\delta_{i\,j},\quad\operatorname{div} \theta_j=0,
\end{equation}
and the second (middle) supremum  is over all   trajectories $u(t)$ on the attractor
$\mathscr A$. Then,
using the cancellation mentioned above together with the pointwise estimate \eqref{dpointwise} proved in Appendix \ref{sB}, we get
\begin{multline}
\sum_{j=1}^n(L_{u(t)}\bar\theta_j,\bar\theta_j)_\alpha=
-\sum_{j=1}^n\gamma\|\bar\theta_j\|^2_{\alpha}-
\sum_{j=1}^n((\bar\theta_j,\Nx)\bar u,\bar\theta_j)\le\\\le
-\gamma n+\sqrt{\frac23}\int_\Omega\rho(x)|\nabla_x\bar u(t,x)|\,dx
\le -\gamma n+\sqrt{\frac23}\|\Nx \bar u(t)\|_{L^2}\|\rho\|_{L^2},
\end{multline}
 where
$$
\rho(x)=\sum_{j=1}^n|\bar\theta_j(x)|^2.
$$
We now use  estimate
 \eqref{R23alpha} from  Appendix \ref{sec5}
\begin{equation}\label{Lieb-for-us}
\|\rho\|_{L^2}\le\frac1{2\sqrt{\pi}}\frac{n^{1/2}}{{\alpha}^{3/4}}
\end{equation}
and obtain
\begin{equation}
\sum_{j=1}^n(L_{u(t)}\bar\theta_j,\bar\theta_j)_\alpha\le-\gamma n+\frac1{\sqrt 6\pi}\frac{n^{1/2}}{\alpha^{3/4}}\| \Nx\bar u(t)\|_{L^2}.
\end{equation}
Finally, using \eqref{2.est-int}, we arrive at
$$
q(n)\le-\gamma n +\frac{1}{2\sqrt{3\pi}}\frac{n^{1/2}}{{\alpha}^{5/4}}
\frac{\| g\|_{L^2}}{\gamma}\,.
$$
It only remains to recall that, according to the general theory,
$\omega_n(\mathscr A)\le q(n)$ and
any number $n^*$  for which $\omega_{n^*}(\mathscr A)\le0$  and
$\omega_{n}(\mathscr A)<0$ for $n>n^*$ is  an  upper bound both for the
Hausdorff \cite{B-V,T} and
the fractal \cite{Ch-I2001,Ch-I} dimension of the global attractor
$\mathscr A$. This gives the desired estimate
$$
\dim_F\mathscr{A}\le
\frac{1}{12\pi}\,\frac{\| g\|_{L^2}^2}{\alpha^{5/2}\gamma^4}
$$
and finishes the proof of the theorem.
\end{proof}
\begin{remark}\label{Rem:2d}
{\rm
Estimates~\eqref{0.est2} and \eqref{0.est3} for $\mathbb{T}^2$  (and the fact that it is sharp)  were proved in~\cite{IZLap70}.
The upper bound for  $\mathbb{R}^2$ is exactly the same once
we now know~\eqref{R23alpha} for $\mathbb{R}^2$. For a bounded domain
we only need to replace in the proof in~\cite{IZLap70} the estimates of the solutions on the attractor
by \eqref{2.est-int}. Alternatively, one can go through the
proof of Theorem~\ref{Th:est} and replace the 3D constants  by their
2D counterparts accordingly.
}
\end{remark}

 \setcounter{equation}{0}
\section{Sharp lower bound on $\mathbb{T}^3$ }\label{sec4}
The aim of this section is to show that estimate~\eqref{dim-est}
for system \eqref{DEalpha} on $\mathbb{T}^3=[0,2\pi]^3$ is
sharp in the limit as $\alpha\to0$.

We consider a family of right-hand sides
\begin{equation}\label{Kolmf}
g=g_s=\begin{cases}g_1=\gamma\lambda(s)\sin sx_3,\\
g_2=0,\\
g_3=0,\end{cases}
\end{equation}
depending only on $x_3$ and parameterized by $s\in \mathbb{N}$,
$s\gg1$. The amplitude function $\lambda(s)$  will be specified in the course
of the proof. Corresponding to the family $g_s$ is the family of
stationary solutions of \eqref{DEalpha}
\begin{equation}\label{u0}
\vec u_0(x_3)=\begin{cases}u_0(x_3)=\lambda(s)\sin sx_3,\\
0,\\
0,\end{cases}
\end{equation}
with $p=0$. In fact,
$$
\bar {\vec u}_0=(1-\alpha\Dx)^{-1}\vec u_0=(\bar u_0,0,0)^T
$$ also depends
only on $x_3$ and therefore $(\bar{\vec u}_0,\Nx)\bar{\vec u}_0=0$.

We now consider system~\eqref{DEalpha} linearized on the stationary
solution \eqref{u0}
\begin{equation}\label{linear}
\left\{
\aligned
    \partial_t w+\bar u_0\frac{\partial\bar w}{\partial x_1}
    +\bar w_3 \frac{\partial\bar u_0}{\partial x_3}e_1+\gamma w+\nabla_x q=0, \\
    \operatorname{div} w=0,
\endaligned
\right.
\end{equation}
where $e_1=(1,0,0)^T$ and  $\bar w=(1-\alpha\Dx )^{-1}w$. The
standing assumption~is
\begin{equation}\label{zero}
    \int_{\mathbb{T}^3}w(x,t)dx=0.
\end{equation}
We shall  look for the solution of the linear problem
\eqref{linear} in the form
\begin{equation}\label{exp}
 w(x,t)=\left(\aligned
 w_1(x_3)\\ w_2(x_3)\\ w_3(x_3)
 \endaligned
 \right)
 e^{i(ax_1+bx_2-act)},\qquad
 q(x,t)=q(x_3)e^{i(ax_1+bx_2-act)},
\end{equation}
where $a,b\in\mathbb{Z}$ so that $w$ and $q$ are $2\pi$-periodic in
each $x_i$. If such a solution of~\eqref{linear} is found, then
substituting~\eqref{exp} into~\eqref{linear} and setting $t=0$ we
see that
$$
w(x,0)=\left(\aligned
 w_1(x_3)\\ w_2(x_3)\\ w_3(x_3)
 \endaligned
 \right)
 e^{i(ax_1+bx_2)}
 $$
is a vector-valued eigenfunction of the stationary operator
\begin{equation}\label{L3}
L_3(\vec u_0)w=\bar u_0\frac{\partial\bar w}{\partial x_3}
    \bar u+\bar w_3 \frac{\partial\bar u_0}{\partial x_3}e_1+\gamma w+\nabla_x q
\end{equation}
and $iac$ is the corresponding eigenvalue. If $\mathrm{Re}(iac)<0$,
then the corresponding mode is unstable.

We substitute \eqref{exp} into \eqref{linear} and obtain the system
\begin{equation}\label{abc}
\left\{
  \begin{array}{ll}
 -\gamma w_1-ia(\bar u_0\bar w_1-c w_1)=iaq+\bar w_3u_0' \\
    -\gamma w_2-ia(\bar u_0\bar w_2-c w_2)=ibq,  \\
     -\gamma w_3-ia(\bar u_0\bar w_3-c w_3)=q',  \\
 ia w_1+ibw_2+w_3'=0,
  \end{array}
\right.
\end{equation}
where $\ '=\partial\,/\partial x_3$.
\begin{lemma}\label{L:a0}
There are no unstable solutions of equation \eqref{linear}
 $$
\partial_tw=L_3(\vec u_0)w
$$
that can be written in the form~\eqref{exp} with $a=0$.
\end{lemma}
\begin{proof}
Let $a=0$. Then the solutions of \eqref{linear} are sought in the
form
$$
 w(x,t)=\left(\aligned
 w_1(x_3)\\ w_2(x_3)\\ w_3(x_3)
 \endaligned
 \right)
 e^{i(bx_2-ct)},\qquad
 q(x)=q(x_3)e^{i(bx_2-ct)},
$$
and \eqref{abc} goes over to
$$
\left\{
  \begin{array}{ll}
 -\gamma w_1+ic w_1=\bar w_3u_0' \\
    -\gamma w_2+ic w_2=ibq,  \\
     -\gamma w_3+ic w_3=q',  \\
 ibw_2+w_3'=0.
  \end{array}
\right.
$$
Let $b\ne0$. Then $w_2=-w_3'/(ib)$. Substituting this into the
second  equation and differentiating the third with respect to
$x_3$ we obtain
$$
q''=b^2q,
$$
which gives that $q=0$, since $q$ is periodic. Since we are looking
for unstable solutions, it follows that $\mathrm{Re}(ic)<0$ and
therefore $-\gamma+ic\ne0$. This gives that $w_2=w_3=0$, and,
finally, $w_1=0$.

If $a=b=0$, then $w_3'=0$, and $w_3=0$ by periodicity and zero mean
condition. This gives $q=0$ and $w_1=w_2=0$. The proof is
complete.
\end{proof}

\subsection{Squire's transformation} We now reduce the 3D
instability analysis to the instability analysis of the transformed
2D problem. The key role is played by the Squire's transformation
(see \cite{Squire}, \cite {DrazinReid}, \cite{Liu2}).

Since we a looking for unstable solutions of \eqref{linear}, in
view of Lemma~\ref{L:a0} we may assume that $a\ne0$ in \eqref{abc}.
Multiplying the first equation in~\eqref{abc} by $a$ and the second
by~$b$ a adding up the results, we obtain
\begin{equation}\label{Squire}
\left\{
  \begin{array}{ll}
 -\widehat\gamma \widehat w_1-i\widehat a(\bar u_0\bar {\widehat{w}_1}-
\widehat c\, \widehat w_1)=i\widehat a\,  \widehat q+\bar{\widehat{w}}_3u_0', \\
-\widehat\gamma \widehat w_3-i\widehat a(\bar u_0\bar {\widehat{w}_3}-
\widehat c\, \widehat w_3)=\widehat q\,',  \\
 i\widehat a\, \widehat w_1+\widehat w'_3=0,
  \end{array}
\right.
\end{equation}
where
\begin{equation}\label{tilde}
\aligned
\widehat a\,^2=a^2+b^2,\qquad &\widehat w_1=\frac{aw_1+bw_2}{\widehat a},\qquad
\widehat w_3=w_3,\\
\widehat\gamma=\gamma\frac{\widehat a}a,\qquad
&\widehat q=q\frac{\widehat a}a,\qquad \widehat c=c.
\endaligned
\end{equation}

The solutions of this problem on the 2d torus
$$\mathbb{T}^2_{|\widehat a|}=x_1\in [0,2\pi/|\widehat a|\,],\
x_3\in[0,2\pi]$$
 are sought in the form
\begin{equation}\label{exp2}
 \widehat w(x_1,x_3,t)=\left(\aligned
 \widehat w_1(x_3)\\ \widehat w_3(x_3)
 \endaligned
 \right)
 e^{i(\widehat a\,x_1-\widehat a\widehat ct)},\quad
 \widehat q(x_1,x_3,t)=q(x_3)e^{i(\widehat a\,x_1-\widehat a\widehat ct)},
\end{equation}
and if such a solution is found, then the vector function
\begin{equation}\label{w-unstable}
\widehat w(x_1,x_3,0)=\left(\aligned
 \widehat w_1(x_3)\\ \widehat w_3(x_3)
 \endaligned
 \right)
 e^{i\widehat a\,x_1}
\end{equation}
is a vector-valued eigenfunction with
 eigenvalue  $i\widehat a\,\widehat c$ of the stationary operator
\begin{equation}\label{L2}
L_2(\vec u_0)\widehat w=\widehat \gamma \widehat w+\bar u_0
\frac{\partial\widehat{\bar w}}{\partial x_1}
    +\widehat{\bar w}_3 \frac{\partial\bar u_0}{\partial x_3}e_1+\Nx\widehat  q,
\quad \operatorname{div}\widehat w=0,
\end{equation}
on  $\mathbb{T}^2_{|\widehat a|}$, where the stationary solution
and the generating right-hand side are
\begin{equation}\label{u02}
\vec u_0(x_3)=\begin{cases}u_0(x_3)=\lambda(s)\sin sx_3,\\
0,\end{cases}\
g_s(x_3)=\begin{cases}g_1(x_3)=\lambda(s)\widehat\gamma\sin sx_3,\\
0,\end{cases}
\end{equation}
and where as before
$\bar u_0=\bar u_0(x_3)=(1-\Dx)^{-1}u_0.$

To avoid unnecessary complications we assume in what follows that
$$
\sqrt{a^2+b^2}=\widehat a>0,\quad\ a>0,
$$
and formulate the main result on the Squire's reduction of the 3D
instability analysis to the 2D case.

\begin{lemma}\label{L:Squire}
Let $\widehat w$ in \eqref{w-unstable} be an unstable eigenfunction of
the opera\-tor~\eqref{L2} on the torus
$\mathbb{T}^2_{\widehat a}=[0,2\pi/\widehat a\,]\times[0,2\pi]$. Then for
any pair of integers $a,b\in\mathbb{Z}$ with
$$
a^2+b^2=\widehat a^2
$$
there exist an unstable solution of system~\eqref{abc} on $\mathbb{T}^3=[0,2\pi]^3$.
\end{lemma}
\begin{proof}
Once the $\widehat\cdot$\,-\,variables are known, $q,w_3,c$ and
$\gamma$ are found from \eqref{tilde}. It remains to find $w_1$ and
$w_2$. We consider the second equation in~\eqref{abc}:
$$
{\rm A}w_2:=(-\gamma +iac)w_2-ia\bar u_0\bar w_2=ibq.
$$
Since $\widehat w$ is unstable, $\mathrm{Re}(iac)<0$ and therefore
$\mathrm{Re}(-\gamma+iac)<0$. Suppose that ${\rm A}w_2=0$ for some $w_2$. Taking
the scalar product in (complex) $L^2(0,2\pi)$
with $w_2$ and taking into account that the second term is purely imaginary
we obtain for the real part
$$
\mathrm{Re}(-\gamma+iac)\|w_2\|_{L^2}^2=0,
$$
which gives that $w_2=0$, and $A$ has a trivial kernel.
In addition, $A$ is a Fredholm operator,  since the second term is
compact (smoothing). Hence it has a bounded inverse,
and since $q$ is known, we have  found $w_2$. Finally,
$$w_1=(\widehat a\widehat w_1-bw_2)/a.$$
\end{proof}

\subsection{Instability analysis on $\mathbb{T}^2$}  We now have to recall the
instability analysis for the 2D problem that was carried out in detail in
our previous work~\cite{IZLap70}.
The problem was studied on the standard torus $\mathbb{T}^2=[0,2\pi]^2$
and we now denote the second coordinate by $x_3$, so that
$x_1$, $x_3$ are the coordinates on $\mathbb{T}^2$. The family of the forcing terms
and the corresponding stationary solutions are as in~\eqref{u02}
and the linearized stationary operator is precisely~\eqref{L2}.
Applying $\operatorname{curl}$ to~\eqref{L2}
 we obtain the equivalent  scalar
operator in terms of the vorticity  whose spectrum was  studied in \cite{IZLap70}
\begin{equation}\label{linvort}
 \aligned{\mathcal L}_{s}\omega:=
    &J\bigl((\Dx-\alpha\Dx^2)^{-1}\omega_s,(1-\alpha\Dx)^{-1}\omega\bigr)+\\+
&J\bigl((\Dx-\alpha\Dx^2)^{-1}\omega,(1-\alpha\Dx)^{-1}\omega_s\bigr)+
\gamma\omega=
-\sigma\omega,
\endaligned
\end{equation}
where
$$
J(a,b)=\nabla a\cdot\nabla^\perp b=\partial_{x_1}a\,\partial_{x_3}b-
\partial_{x_3} a\,\partial_{x_1} b,
$$
and
$$
\omega_s=\operatorname{curl}\vec u_0=-\lambda(s)s\cos sx_2,\qquad
\omega=\operatorname{curl}\widehat w.
$$
The following result was proved in \cite{IZLap70} (see Theorem~4.1 and Corollary~4.2.)
\begin{theorem}\label{Th:2d}
Given a large  integer $s>0$
let a fixed pair of integers $t,r$ belong to a bounded region
$A(\delta)$ defined by conditions
\begin{equation}\label{cond}
t^2+r^2<s^2/3,\quad t^2+(-s+r)^2>s^2,\quad t^2+(s+r)^2>s^2, \quad t\ge\delta s,
\end{equation}
where $0<\delta<1/\sqrt{3}$.
There exists an absolute constant $c_1$ such that
for
\begin{equation}\label{lambda}
\lambda\ge\lambda_2(s,\gamma)=c_1 \gamma\frac{(1+\alpha s^2)^2}s
\end{equation}
in~\eqref{u02} the linear operator ${\mathcal L}_{s}$ on the torus
$\mathbb{T}^2=[0,2\pi]^2$ has a real
negative (unstable) eigenvalue $\sigma<0$ of multiplicity $2$.
The corresponding eigenfunctions are
\begin{equation}\label{omega}
\aligned
&\omega^1(x_1,x_3)=\sum_{n=-\infty}^\infty a_{t,sn+r}\cos(tx_1+(sn+r)x_3),\\
&\omega^2(x_1,x_3)=\sum_{n=-\infty}^\infty a_{t,sn+r}\sin(tx_1+(sn+r)x_3).
\endaligned
\end{equation}
\end{theorem}

We now  observe that $\omega^1$ and $\omega^2$ in~\eqref{omega} are the real and imaginary parts of the
complex-valued eigenfunction
$$
\omega^1(x_1,x_3)+i\omega^2(x_1,x_3)=
\left[\sum_{n=-\infty}^\infty a_{t,sn+r}e^{i(sn+r)x_3}\right]e^{itx_1}.
$$

Recovering the corresponding divergence free vector function, that is, applying the
operator $\Nx^\perp\Dx^{-1}$, we obtain an unstable vector valued eigenfunction
of the operator $L_2(\vec u_0)$ written in the required form \eqref{w-unstable}:
$$
 w(x_1,x_3)=\left(\aligned
  w_1(x_3)\\ w_3(x_3)
 \endaligned
 \right)
 e^{it\,x_1}.
$$

For the 3D instability analysis below we need to repeat the
construction of an unstable eigenmode on the torus
$\mathbb{T}^2_\varepsilon$ with $x_1\in[0,2\pi/\varepsilon],x_2\in[0,2\pi]$,
where $\varepsilon>0$ is arbitrary (not necessarily small).
\begin{proposition}\label{Prop:eps}
Let $r$ and $t':=t\varepsilon$ belong to the region $A(\delta)$:
\begin{equation}\label{cond1}
{t'}^2+r^2<s^2/3,\quad {t'}^2+(-s+r)^2>s^2,\quad {t'}^2+(s+r)^2>s^2,
\quad {t'}\ge\delta s.
\end{equation}
Let $\lambda$ be defined in \eqref{lambda} and let $g_s$ and $\vec u_0$ be
the same as before but in two dimensions:
$$
g_s(x_3)=(\gamma\lambda(s)\sin sx_3,0)^T,\quad
\vec u_0(x_3)=(\lambda(s)\sin sx_3,0)^T.
$$
Then there exists an unstable solution
\begin{equation}\label{veceigmode}
 w(x_1,x_3)=\left(\aligned
  w_1(x_3)\\ w_3(x_3)
 \endaligned
 \right)
 e^{it\varepsilon\,x_1},\qquad x\in \mathbb{T}^2_\varepsilon.
\end{equation}
of the form \eqref{w-unstable} of the operator~\eqref{L2}
on the torus $ \mathbb{T}^2_\varepsilon.$
\end{proposition}
\begin{proof}
Following the proof of Theorem~4.1 in \cite{IZLap70} we see that a word for word
repetition of it shows that if  $t'=\varepsilon t ,r$
satisfy~\eqref{cond1}, then the corresponding operator \eqref{linvort}
has an unstable (real negative) eigenvalue of multiplicity two
with eigenfunctions
$$
\aligned
&\omega^1(x_1,x_3)=\sum_{n=-\infty}^\infty a_{t,sn+r}\cos(t\varepsilon x_1+(sn+r)x_3),\\
&\omega^2(x_1,x_3)=\sum_{n=-\infty}^\infty a_{t,sn+r}\sin(t\varepsilon x_1+(sn+r)x_3),
\endaligned
$$
from which we construct the required vector valued complex eigenfunction
\eqref{veceigmode} as before.
\end{proof}

It is convenient for us to single out a small rectangle $D$ in
the $(t',r)$-plane  inside the
region $A(\delta)$ defined by~\eqref{cond1}, see Fig.\ref{fig1}:
\begin{equation}\label{rectan}
|r|\le c_2s,\quad 0<c_3s\le t'=t\varepsilon\le c_4s.
\end{equation}
Here $\delta=\delta^*\in(0,1/\sqrt{3})$ is fixed, and all the constants $c_i$ are absolute constants,
 whose explicit values can easily be specified.

\begin{figure}[htb]
\def\svgwidth{12.5cm}
\def\svgheight{9cm}
\begingroup%
  \makeatletter%
  \providecommand\color[2][]{%
    \errmessage{(Inkscape) Color is used for the text in Inkscape, but the package 'color.sty' is not loaded}%
    \renewcommand\color[2][]{}%
  }%
  \providecommand\transparent[1]{%
    \errmessage{(Inkscape) Transparency is used (non-zero) for the text in Inkscape, but the package 'transparent.sty' is not loaded}%
    \renewcommand\transparent[1]{}%
  }%
  \providecommand\rotatebox[2]{#2}%
  \ifx\svgwidth\undefined%
    \setlength{\unitlength}{419.00001526bp}%
    \ifx\svgscale\undefined%
      \relax%
    \else%
      \setlength{\unitlength}{\unitlength * \real{\svgscale}}%
    \fi%
  \else%
    \setlength{\unitlength}{\svgwidth}%
  \fi%
  \global\let\svgwidth\undefined%
  \global\let\svgscale\undefined%
  \makeatother%
  \begin{picture}(1,0.75178995)%
    \put(0,0){\includegraphics[width=\unitlength]{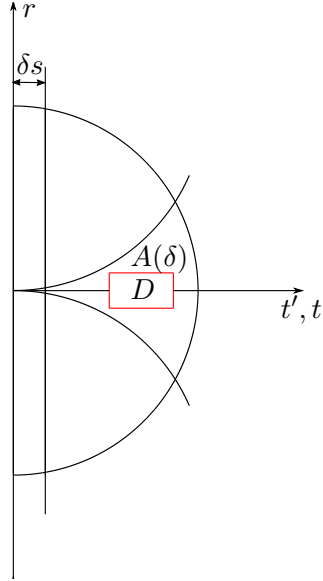}}%
    \put(0.369,0.62){\color[rgb]{0,0,0}\makebox(0,0)[lb]{\smash{$\delta s$}}}%
    \put(0.375,0.68){\color[rgb]{0,0,0}\makebox(0,0)[lb]{\smash{$r$}}}%
    \put(0.65,0.36){\color[rgb]{0,0,0}\makebox(0,0)[lb]{\smash{$t',t$}}}%
    \put(0.49,0.415){\color[rgb]{0,0,0}\makebox(0,0)[lb]{\smash{$A(\delta)$}}}%
    \put(0.49,0.38){\color[rgb]{0,0,0}\makebox(0,0)[lb]{\smash{$D$}}}%
  \end{picture}%
\endgroup%
\caption{The region $A(\delta)$ and the rectangle $D$.}
\label{fig1}
\end{figure}

\subsection{3D lower bound} We can now formulate the main results of this
section.
\begin{theorem}\label{Th:3d}
We consider the linearized system on the 3D torus $\mathbb{T}^3=[0,2\pi]^3$
with right-hand side $g_s$ and stationary solution $\vec u_0$ given in
\eqref{Kolmf} and \eqref{u0}, where
\begin{equation}\label{lambda3}
\lambda=\lambda_3(s)=\sqrt{2}\lambda_2(s,\gamma)=\sqrt{2}c_1\gamma\frac{(1+\alpha s^2)^2}s.
\end{equation}
Then for each triple of integers $a,b,r$ satisfying
\begin{equation}\label{abr}
\aligned
c_3s\le\widehat a=&\sqrt{a^2+b^2}\le c_4s, \quad |r|\le c_2s ,\\
&a\ge|b|,
\endaligned
\end{equation}
there exists an unstable solution of the linearized operator~\eqref{linear}.
\end{theorem}
\begin{proof}
We fix $a,b,r$ satisfying~\eqref{abr}.
Then in view of the first two inequalities in \eqref{abr}
the pair $(t',r)\in D\subset A(\delta)$, where $t'=\widehat a\cdot 1$
(so that we set $t=1$).
Applying Squire's transformation we obtain a 2D linearized problem on
the torus $\mathbb{T}^2_{\widehat a }$ of the form~\eqref{L2}
with $\widehat \gamma=\gamma\widehat a/a$. Then  in view of the third
inequality in  \eqref{abr} we have
$$
\lambda=
\sqrt{2}\lambda_2(s, \gamma)=
\lambda_2(s,\sqrt{2} \widehat \gamma\, a/\widehat a)\ge
\lambda_2(s,\widehat \gamma).
$$

We now see from  Proposition~\ref{Prop:eps} that
the 2D linearized problem~\eqref{L2} has an unstable
eigenvalue. Lemma~\ref{L:Squire} says, in turn, that so does
the linearized problem~\eqref{L3} on the standard torus $\mathbb{T}^3=[0,2\pi]^3$.
\end{proof}

The number of integers $(a,b,r)$ satisfying~\eqref{abr}
is of order $c_5s^3$, where
$$
c_5=\frac14\pi c_2(c_4^2-c_3^2).
$$

\begin{theorem}\label{Th:3d-lbd}
Let the right-hand side in~\eqref{DEalpha} be $g_s$ defined in~\eqref{Kolmf}
with $\lambda(s)$ defined in \eqref{lambda3}. Then the dimension of the
corresponding attractor $\mathscr{A}=\mathscr{A}_s$
satisfies for an  absolute constant $c_6$ the lower bound
\begin{equation}\label{3d-lbd}
\dim_F\mathscr{A}\ge c_6\frac{\| g_s\|_{L^2}^2}{\alpha^{5/2}\gamma^4}\,.
\end{equation}
\end{theorem}
\begin{proof}
We study our system in the limit  $\alpha\to 0$. Since $s$ is at our disposal we~set
$$
s=\frac1{\sqrt{\alpha}}\,.
$$
Then $\lambda$ in \eqref{lambda3} and  $\|g_s\|_{L^2}^2$ in~\eqref{Kolmf}
become
$$
\lambda=c_7\gamma{\sqrt{\alpha}},\qquad
\|g_s\|_{L^2}^2=c_8{\gamma^4}\alpha.
$$
We can finally write
$$
\dim_F\mathscr{A}\ge c_5s^3=c_5\frac1{\alpha^{3/2}}=
c_5\frac{\alpha \|g_s\|_{L^2}^2}{\alpha^{5/2}\|g_s\|_{L^2}^2}=
c_5\frac{\alpha \|g_s\|_{L^2}^2}{\alpha^{5/2}c_8\alpha\gamma^4},
$$
which gives~\eqref{3d-lbd} with $c_6=c_5/c_8$.
\end{proof}
\begin{remark}\label{Rem:lower2d}
{\rm
We recall  the sharp lower bound
in~\cite{IZLap70}  for the attractor dimension
for system~\eqref{DEalpha} on the 2D torus $\mathbb{T}^2$:
\begin{equation}\label{lower2d}
\dim_F\mathscr{A}\ge c_{\mathrm{abs}}
\frac{\|\operatorname{curl}g_s\|_{L^2}^2}{\alpha\gamma^4}\,.
\end{equation}
Since for $\lambda\sim\gamma{\sqrt{\alpha}}$ we have
$\|\operatorname{curl}g_s\|_{L^2}^2\sim\gamma^4$
and $\|g_s\|_{L^2}^2\sim{\gamma^4}\alpha$,
it follows that estimate~\eqref{lower2d} can equivalently be written
in terms of the  other dimensionless number in \eqref{0.est4} as follows
$$
\dim_F\mathscr{A}\ge c'_{\mathrm{abs}}
\frac{\|g_s\|_{L^2}^2}{\alpha^2\gamma^4}\,.
$$
}
\end{remark}

\setcounter{equation}{0}
\appendix
\section{Collective Sobolev   inequalities for $H^1$-orthonormal families}\label{sec5}
We prove here some spectral inequalities for orthonormal families of functions
which are the key technical tool in our derivation of sharp upper bounds for
the attractor dimension. These inequalities are somehow complementary
to the classical Lieb--Thirring inequalities for orthonormal
systems in $L^2$ and in our case we estimate the proper norms
of the same quantity $\rho(x)=\sum_{i=1}^n|\bar\theta_i(x)|^2$,
but  for   families $\{\bar\theta_i\}_{i=1}^n$ that are orthonormal in $H^1$
with  norm \eqref{2.norm} depending on $\alpha$.
Our exposition utilizes the ideas
from \cite{LiebJFA} as well as extends the results of
\cite{IZLap70} to 3D case.
\par
We start with the case $\Omega=\R^d$ or $\Omega\subset\R^d$  with
Dirichlet boundary conditions (analogously to the classical
Lieb--Thirring inequality, the case of Dirichlet boundary
conditions follows from the case of whole space by zero extension).

\subsection{The case of the whole space and a domain with Dirichlet BC}

\begin{theorem}\label{Th:Rddiv=0}
Let $\Omega\subseteq\mathbb{R}^d$ be an arbitrary domain.
Let a family of vector functions $\{\bar\theta_i\}_{i=1}^n\in {\bf H}^1(\Omega)$  with
$\divv\bar\theta_i=0$ be
orthonormal with respect to the scalar product
\begin{equation}\label{H1orth}
m^2(\bar\theta_i,\bar\theta_j)_{L^2}+
(\nabla\bar\theta_i,\nabla\bar\theta_j)_{L^2}=
m^2(\bar\theta_i,\bar\theta_j)_{L^2}+
(\operatorname{curl}\bar\theta_i,\operatorname{curl}\\
\bar\theta_j)_{L^2}=
\delta_{ij},
\end{equation}
 Then the function $\rho(x):=\sum_{j=1}^n|\bar\theta_j(x)|^2$
satisfies
\begin{equation}\label{R23}
\aligned
&\|\rho\|_{L^2}\le\frac1{2\sqrt{\pi}}\frac{n^{1/2}}{m},\qquad d=2,\\
&\|\rho\|_{L^2}\le\frac1{2\sqrt{\pi}}\frac{n^{1/2}}{m^{1/2}}, \qquad d=3.
\endaligned
\end{equation}
\end{theorem}
\begin{proof}
We first let $\Omega=\mathbb{R}^d$ and    introduce the operators
$$
\Bbb H= V^{1/2}(m^2-\mathbf{\Delta}_x)^{-1/2}\Pi,\quad \Bbb H^*=\Pi(m^2-\mathbf{\Delta}_x)^{-1/2}V^{1/2}
$$
acting in $[L^2(\R^d)]^d$, where $V\in L^2(\R^d)$ is a non-negative scalar function
which will be specified below and $\Pi$ is the Helmholtz--Leray projection to
divergence free vector fields.  Then ${\bf K}=\Bbb H^*\Bbb H$ is a  compact
self-adjoint operator acting from $[L^2(\R^d)]^d$ to $[L^2(\R^d)]^d$ and
$$
\aligned
\Tr \mathbf{K}^2=&\Tr\left(\Pi(m^2-\Dx)^{-1/2}V(m^2-\Dx)^{-1/2}\Pi\right)^2\le\\&\le
\Tr\left(\Pi(m^2-\Dx)^{-1}V^2(m^2-\Dx)^{-1}\Pi\right)=\\&=
\Tr\left(V^2(m^2-\mathbf{\Delta}_x)^{2}\Pi\right),
\endaligned
$$
where we used
 the Araki--Lieb--Thirring inequality for traces \cite{Araki,  LT, traceSimon}:
$$
\Tr(BA^2B)^p\le\Tr(B^pA^{2p}B^p),\quad p\ge1,
$$
and the cyclicity property of the trace together with the facts that $\Pi$
commutes with the Laplacian and that $\Pi$ is a projection: $\Pi^2=\Pi$.
 \par
We want to show that
\begin{equation}\label{TrFund}
\Tr \mathbf{K}^2\le\left\{
                     \aligned
                      &\frac1{4\pi}\frac1{m^2}\|V\|_{L^2}^2,  &\hbox{$d=2$;} \\
                       &\frac1{4\pi}\frac1{m}\|V\|_{L^2}^2,  &\hbox{$d=3$.}
                     \endaligned
                   \right.
\end{equation}
Indeed, the fundamental solution of $(m^2-\Dx)^{2}\Pi$ in $\R^d$
is a $d\times d$ matrix
$$
\mathbf{F}^d_{ij}(x)=G_d(x)\delta_{ij}-\partial_{x_i}\partial_{x_j}\Delta^{-1}G_d(x)
$$
with $\mathbb{R}^d$-trace at $x\in \mathbb{R}^d$
$$
\Tr_{\mathbb{R}^d}\mathbf{F}^d(x)=dG_d(x)-
\sum_{i=1}^d\partial^2_{x_ix_i}\Dx^{-1}G_d(x)=(d-1)G_d(x),
$$
where $G_d(x)$ is the fundamental solution of the scalar operator $(m^2-\Dx)^2$ in the whole space $\R^d$:
\begin{equation}
\label{Fund}
G_d(x)=\frac1{(2\pi)^d}\int_{\mathbb{R}^d}\frac{e^{i\xi x}d\xi}{(m^2+|\xi|^2)^2}=
\left\{
\aligned
&\frac1{8\pi}\frac1m e^{-|x|m}, & \hbox{$d=3$;} \\
&\frac 1{4\pi}\frac1{m^2}|x|mK_1(|x|m), & \hbox{$d=2$.}
 \endaligned\right.
\end{equation}
The first equality here follows from \eqref{Fexp}, while for the second we
have (since the function is radial and using formula 13.51(4) in
\cite{Watson})
$$
G_2(x)\!=\!\frac1{2\pi}\mathscr{F}^{-1}\bigl((m^2+|\xi|^2)^2\bigr)\!=\!
\frac1{2\pi}\int_0^\infty\frac{J_0(|x|r)rdr}{(m^2+r^2)^2}\!=\!
\frac 1{4\pi}\frac1{m^2}|x|mK_1(|x|m),
$$
where $K_1$ is the modified Bessel function of the second kind.
\par
Thus, the operator $V^2(m^2-\Dx)^{2}\Pi$ has the matrix-valued
integral kernel
$$
V(y)^2\mathbf{F}^d(x-y)
$$
and therefore
\begin{multline}\label{keytrace}
\Tr(V^2(m^2-\Dx)^{2}\Pi)=\\=\int_{\mathbb{R}^d}\Tr_{\mathbb{R}^d}\bigl(V(y)^2\,\mathbf{F}^d(0)\bigr)dy=
(d-1)\|V\|_{L^2}^2G_d(0)
\end{multline}
which along with \eqref{Fund} proves the first inequality in~\eqref{TrFund},
and also the second one, since $(tK_1(t))\vert_{t=0}=1$.

We can now complete the proof as in~\cite{LiebJFA}.
Setting
$$
\psi_i:=(m^2-\Dx)^{1/2}\bar\theta_i,
$$
we see from \eqref{H1orth} that $\{\psi_j\}_{j=1}^n$ is an orthonormal family in $L^2$.
We observe that
$$
\int_{\mathbb{R}^d}\rho(x)V(x)dx=\sum_{i=1}^n\|\Bbb H\psi_i\|^2_{L^2}.
$$
By the orthonormality of the $\psi_j$'s in $L^2$ and the definition
of the trace we obtain
\begin{multline}
\sum_{i=1}^n\|\Bbb H\psi_i\|^2_{L^2}=\sum_{i=1}^n({\bf K}\psi_i,\psi_i)\le
\sum_{i=1}^n\|{\bf K}\psi_i\|_{L^2}\le n^{1/2}
\left(\sum_{i=1}^n\|{\bf K}\psi_i\|_{L^2}^2\right)^{1/2}=\\= n^{1/2}
\(\sum_{i=1}^n({\bf K}^2\psi_i,\psi_i)\)^{1/2}
\le n^{1/2}\left(\Tr {\bf K}^2\right)^{1/2}.
\end{multline}
This gives
$$
\int_{\mathbb{R}^d}\rho(x)V(x)dx\le n^{1/2}\left(\Tr {\bf K}^2\right)^{1/2}.
$$
Setting $V(x):=\rho(x)$ and using \eqref{TrFund}, we complete the proof of
\eqref{R23} for the case of $\Omega=\R^d$, $d=2,3$.
\par
Finally, if $\Omega$ is a proper domain in $\mathbb{R}^d$,
we extend by zero the vector functions $\bar\theta_j$ outside
$\Omega$ and denote the results by $\widetilde{\bar\theta}_j$,
so that $\widetilde{\bar\theta}_j\in {\bf H}^1(\R^d)$
and $\operatorname{div}\widetilde{\bar\theta}_j=0$. We
further set
$\widetilde\rho(x):=\sum_{j=1}^n|\widetilde{\bar\theta}_j(x)|^2$.
Then setting $\widetilde\psi_i:=(m^2-\Dx)^{1/2}\widetilde{\bar\theta}_i$,
we see that the system $\{\widetilde\psi_j\}_{j=1}^n$ is orthonormal in $L^2(\mathbb{R}^d)$ and
$\operatorname{div}\widetilde\psi_j=0$.
Since clearly $\|\widetilde\rho\|_{L^2(\R^d)}=\|\rho\|_{L^2(\Omega)}$,
the proof of estimate~\eqref{R23} reduces to the case of $\R^d$
and therefore is complete.
\end{proof}
The proved result can be rewritten in terms of orthogonal functions with respect to inner product \eqref{scal-alpha} as follows.

\begin{corollary}\label{Cor:mtoalpha2} Let the assumptions
of Theorem \ref{Th:Rddiv=0} hold and let $\{\bar\theta_j\}_{j=1}^n$,
$\operatorname{div}\bar\theta_j=0$
be an orthonormal system  with respect~to
\begin{equation}\label{alpha-ort}
(\bar\theta_i,\bar\theta_j)_{L^2}+\alpha
(\nabla\bar\theta_i,\nabla\bar\theta_j)_{L^2}=\delta_{ij}.
\end{equation}
Then $\rho(x)=\sum_{j=1}^n|\bar\theta_j(x)|^2$ satisfies
\begin{equation}\label{R23alpha}
\aligned
\|\rho\|_{L^2}\le\frac1{2\sqrt{\pi}}\frac{n^{1/2}}{\alpha^{1/2}},\qquad d=2,\\
\|\rho\|_{L^2}\le\frac1{2\sqrt{\pi}}\frac{n^{1/2}}{\alpha^{3/4}}, \qquad d=3.
\endaligned
\end{equation}
\end{corollary}
Indeed, this statement follows from \eqref{R23} by the proper scaling.

\subsection{The case of periodic BC: Estimates for  the lattice sums} We now turn to
the case $\Omega=\Bbb T^d$. In this case, we naturally have an extra condition that
the considered functions have zero mean. Analogously to the case $\Omega=\R^d$, the
Laplacian commutes with the Helholtz--Leray projection, so we may define analogously
the operator ${\bf K}$ and get exactly the same  expression
\eqref{keytrace} for its trace. The only difference is that now $G_d(x)=G_{d,m}(x)$
 a fundamental solution of the scalar operator $(m^2-\Dx)^{-2}$ on the torus $\Bbb T^d$
 (with zero mean condition), so the integral in \eqref{Fund} should be replaced by
 the corresponding sum over the
 lattice $\Bbb Z_0^d=\Bbb Z^d\setminus\{0\}$:
 \begin{equation}\label{green}
 G_d(x)=\frac1{(2\pi)^d}\sum_{k\in\Bbb Z_0^d}\frac{e^{ik\cdot x}}{(m^2+|k|^2)^2}.
 \end{equation}
Thus, in order to get estimates \eqref{R23} for the torus $\Bbb T^2$ arguing as in
the proof of Theorem \ref{Th:Rddiv=0}, we only need to check that
\begin{equation}\label{est-lat}
G_{d,m}(0)<\left\{
\aligned\frac1{8\pi}\frac1m,\ \ d=3;\\\frac1{4\pi}\frac1{m^2},\ \ d=2.
\endaligned
\right.
\end{equation}
for all $m\ge0$. Unfortunately, we do not know explicit
expressions for the sum \eqref{green}, so we need to do rather
accurate estimates of the associated lattice sum in order get
\eqref{est-lat} based on the Poisson summation formula.
In the case $d=2$, \eqref{est-lat} is proved in \cite{IZLap70} and the
case $d=3$ is considered in Proposition \ref{L:sumZ3} below. Thus,
the following result holds.
\begin{theorem}\label{Th:A-L-T-T3}
 Let a family of divergence free vector functions   with zero mean
 $\{\bar\theta_i\}_{i=1}^n\in\dot
{\mathbf{H}}^1(\mathbb{T}^3)$ be  orthonormal with respect to
scalar product \eqref{H1orth}. Then estimates \eqref{R23} hold.
Analogously, if this family is orthonormal with respect to \eqref{alpha-ort},
then  $\rho$ satisfies inequalities \eqref{R23alpha}.
\end{theorem}
As explained before, the case $d=2$ is verified in \cite{IZLap70} and for
proving the result for $d=3$, it is sufficient to prove the following proposition.

\begin{proposition}\label{L:sumZ3}
The following inequality holds for all $m\ge0$:
\begin{equation}\label{A.1}
F(m):=m\sum_{k\in\Bbb Z^3_0}\frac1{(|k|^2+m^2)^2}<\pi^2
\end{equation}
\end{proposition}
\begin{proof}
Before we go over to the proof, we first observe that
in $\mathbb{R}^3$ we have the equality
$$
\int_{\mathbb{R}^3}\frac{dx}{(|x|^2+m^2)^2}=\frac{\pi^2}m
$$
and, secondly, it is the absence of the term with $k=0$ in the sum in
\eqref{A.1} that makes inequality~\eqref{A.1} hold at all.

We use the Poisson summation formula (see,
e.\,g., \cite{S-W})
$$
\sum_{k\in\mathbb{Z}^n}f(k/m)=
(2\pi)^{n/2}m^n
\sum_{k\in\mathbb{Z}^n}\widehat{f}(2\pi k m),
$$
where
$\mathscr{F}(f)(\xi)=\widehat{f}(\xi)=(2\pi)^{-n/2}\int_{\mathbb{R}^n}
f(x)e^{-i\xi x}dx$. For the function
\begin{equation}\label{Fexp}
f(x)=1/(1+|x|^2)^{2}\  x\in\mathbb{R}^3\quad \text{with}\quad
 \widehat f(\xi)=\frac{\pi^2}{(2\pi)^{3/2}} e^{-|\xi|}
\end{equation}
(see \cite{S-W}), and with
$\int_{\mathbb{R}^3}f(x)dx=\pi^2$
this gives
 \begin{equation}\label{A.2}
 \aligned
F(m)=\frac1{m^3}\sum_{k\in\mathbb{Z}^3}f(k/m)-\frac1{m^3}=\\=\pi^2\sum_{k\in\Bbb Z^3}e^{-2\pi m|k|}-\frac1{m^3}=
\pi^2+\pi^2\sum_{k\in\Bbb Z^3_0}e^{-2\pi m|k|}-\frac1{m^3}\,.
\endaligned
 \end{equation}
In particular, this formula gives a convenient way to compute
$F(m)$ numerically for the case
 where $m$ is not very small. We start the proof of inequality \eqref{A.1} with the case $m\le1$.
 \begin{lemma}\label{LA.1} Inequality \eqref{A.1} holds for all $m\in[0,1]$.
 \end{lemma}
\begin{proof} Since
$$
F'(m)=\sum_{k\in\Bbb Z^3_0}\(\frac1{(|k|^2+m^2)^2}-\frac{4m^2}{(|k|^2+m^2)^3}\)=
\sum_{k\in\Bbb Z^3_0}\frac{|k|^2-3m^2}{(|k|^2+m^2)^3},
$$
we see that all terms in the sum for $F(m)$ with $|k|^2\ge3$ are
monotone increasing (after multiplying by $m$)
 with respect to $m\le1$, so we may write
$$
\aligned
 F(m)\le \frac{6m}{(1+m^2)^2}+\frac{12m}{(2+m^2)^2}+\sum_{|k|^2\ge 3}\frac1{(|k|^2+1)^2}\le\\\le
 \max_{m\in[0,1]}\left\{\frac{6m}{(1+m^2)^2}\right\}+
 \max_{m\in[0,1]}\left\{\frac{12m}{(2+m^2)^2}\right\}+F(1)-\frac 64-\frac{12}9=\\=
 \frac{9\sqrt3}{8}+\frac{9\sqrt{6}}{16}-\frac32-\frac43+\pi^2(1.01306)-1=9.4915
 <\pi^2=9.8696,
\endaligned
$$
where we have used \eqref{A.2} in order to compute
$F(1)=\pi^2(1.01306)-1$ (the calculations are reliable since the
series has an exponential rate of convergence). Thus, the lemma is
proved.
\end{proof}
We now turn to the case $m\ge1$.
\begin{lemma} Inequality \eqref{A.1} holds for all $m\ge1$.
\end{lemma}
\begin{proof} It follows from \eqref{A.2} that
inequality~\eqref{A.1} goes over to
$$
G(m):=\pi^2 m^3\sum_{k\in\Bbb Z^3_0}e^{2\pi
m|k|}-1<0.
$$
We use the inequality
$$
|k|\ge \frac1{\sqrt3}(|k_1|+|k_2|+|k_3|)
$$
for all terms with $|k|>1$ and leave the first $6$ terms with
$|k|=1$ unchanged. This gives
$$
G(m)\le \pi^2 m^3\!\sum_{k\in\Bbb
Z^3_0}e^{2\pi m(|k_1|+|k_2|+|k_3|)/\sqrt3}-1+
6\pi^2m^3\(e^{-2\pi m}-e^{-2\pi m/\sqrt3}\)
$$
and we only need to prove that the right-hand side of this
inequality is negative. Summing the geometric progression, we
get
$$
\aligned
G(m)\le G_0(m):= \pi^2m^3\(\(1+\frac2{e^{2\pi
m/\sqrt3}-1}\)^3-1\)-1+\\+6\pi^2m^3\(e^{-2\pi m}-e^{-2\pi
m/\sqrt3}\)= 6\pi^2m^3\(\frac1{e^{2\pi m/\sqrt3}-1}-e^{-2\pi
m/\sqrt3}\)+\\+12\pi^2\(\frac{m^{3/2}}{e^{2\pi m/\sqrt3}-1}\)^2+
8\pi^2\(\frac m{e^{2\pi m/\sqrt3}-1}\)^3+6\pi^2m^3e^{-2\pi m}-1=\\=
6\pi^2\psi_1(m)+12\pi^2\psi_2(m)^2+8\pi^2\psi_3(m)^3+\psi_4(m).
\endaligned
$$
We claim that all functions $\psi_i(m)$ are monotone decreasing for
$m\ge1$. Indeed, the function $\psi_3(m)$ is obviously decreasing
for all $m\ge0$. The function $\psi_4(m)$ is decreasing for
$m\ge\frac3{2\pi}<1$. Analogously, as elementary calculations show,
the second function is decreasing for $m\ge m_2<1$ where
$$
m_0=\frac{\sqrt3}{4\pi}\(3+2W\(-3e^{-3/2}/2\)\)\approx 0.241,
$$
where $W$ is a Lambert $W$-function.
\ Finally, let us
prove the monotonicity of $\psi_1(m)$. Indeed,
$$
\psi_1'(m)=m^2\frac{2m\pi\sqrt{3}e^{-2\pi m\sqrt{3}/3}-4\pi m\sqrt{3}-9e^{-2\pi m\sqrt{3}/3}+9}
{3(e^{2\pi m\sqrt{3}/3} - 1)^2}
$$
and we see that
$$
\aligned
2m\pi\sqrt{3}e^{-2\pi m\sqrt{3}/3}-4\pi m\sqrt{3}-9e^{-2\pi
m\sqrt{3}/3}+9<\\< 2\pi m\sqrt3\(e^{-2\pi
m\sqrt{3}/3}-1\)+9-2\pi\sqrt3<0
\endaligned
$$
if $m\ge1$, since $9-2\pi\sqrt3<0$. Thus,
$\psi'_1(m)<0$ for $m\ge1$ and
 $\psi_1(m)$ is also decreasing. Thus, $G_0(m)$ is decreasing for $m\ge1$ and we only need to note that
$G_0(1)=-0.7562<0$ and the lemma is proved.
\end{proof}
Finally, we have verified  \eqref{A.1} for all $m\ge0$ and the proof is complete.
\end{proof}
\begin{remark}{\rm Of course, the estimates obtained above
 hold for families of scalar functions $\{\bar\theta_i\}_{i=1}^n\in H^1$
 that are orthonormal with respect to \eqref{alpha-ort}. In this case,
 the factor $(d-1)$ in formula
 \eqref{keytrace} is replaced by $1$, and we get a $\sqrt{2}$-times better constant in
 the   3D  case and the same constant in the 2D case. Namely, the function
 $\rho(x):=\sum_{i=1}^n|\bar\theta_i(x)|$  satisfies
\begin{equation}\label{R23alpha-sc}
\aligned
\|\rho\|_{L^2}\le\frac1{2\sqrt{\pi}}\frac{n^{1/2}}{\alpha^{1/2}},\qquad d=2,\\
\|\rho\|_{L^2}\le\frac1{\sqrt{8\pi}}\frac{n^{1/2}}{\alpha^{3/4}}, \qquad d=3.
\endaligned
\end{equation}
These estimates also hold for all three cases $\Omega=\Bbb T^d$, $\Omega=\R^d$,
and $\Omega\subset\R^d$  with Dirichlet boundary conditions.
}
\end{remark}
\section{A pointwise estimate for the nonlinear term}\label{sB}
In this appendix, we prove a  pointwise estimate for the inertial term
which corresponds to the Navier--Stokes nonliearity.
\begin{proposition}\label{Prop: pointwise}
Let for some $x\in\R^d$, $u(x)\in\R^d$ and
$\div u(x)=0$. Then
\begin{equation}\label{dpointwise}
|\left((\theta, \nabla_x)u,\theta\right)(x)|\le \sqrt{\frac{d-1}d}\,
|\theta(x)|^2|\nabla_x u(x)|,
\end{equation}
where  $\Nx u(x)$ is a $d\times d$ matrix with entries $\partial_iu_j$, and
$$
|\nabla_x u|^2=\sum_{i,j=1}^d(\partial_iu_j)^2.
$$
\end{proposition}
\begin{proof} Basically, this can be extracted from \cite{Lieb}. For the sake
of completeness we reproduce the details.
We suppose first that $A$ is a symmetric real $d\times d$ matrix
with entries $a_{ij}$ and  with $\Tr A=0$.
Then
\begin{equation}\label{forA}
\|A\|_{\mathbb{R}^d\to\mathbb{R}^d}^2\le\frac{d-1}d\sum_{i,j=1}^da_{ij}^2.
\end{equation}
In fact, let $\lambda_1,\dots,\lambda_d$ be the eigenvalues of $A$ and let
$\lambda_1$ be the largest one in absolute value. Then $\sum_{j=1}^d\lambda_j=0$ and
therefore
$$
(d-1)\sum_{j=2}^d\lambda_j^2\ge\left(\sum_{j=2}^d\lambda_j\right)^2=\lambda_1^2.
$$
Adding $(d-1)\lambda_1^2$ to both sides we obtain
$$
(d-1)\sum_{j=1}^d\lambda_j^2\ge  d \lambda_1^2,
$$
which gives \eqref{forA} since $\lambda_1^2=\|A\|_{\mathbb{R}^d\to\mathbb{R}^d}^2$
and $\sum_{j=1}^d\lambda_j^2=\Tr A^2=\sum_{i,j=1}^d a_{ij}^2$.
Now \eqref{dpointwise} follows from \eqref{forA} with $A:=\frac12(\nabla_xu+\nabla_xu^T)$
and $\Tr A=0$, since
$$
\sum_{i,j=1}^da_{ij}^2=\frac14\sum_{i,j=1}^d(\partial_iu_j+\partial_ju_i)^2
\le\sum_{i,j=1}^d(\partial_iu_j)^2.
$$
\end{proof}


\begin{thebibliography}{99}

\bibitem{Araki}
H. Araki,
{\it On an inequality of Lieb and Thirring}.
\emph{Lett. Math. Phys.}, vol 19, no 2 (1990) 167--170.

\bibitem{B-V}
A. Babin and M. Vishik,
 \emph{Attractors of Evolution Equations.}
 Studies in Mathematics and its Applications, vol 25.
 North-Holland Publishing Co., Amsterdam, 1992.
\bibitem{ball}
J. Ball. {\it Global attractors for damped semilinear wave equations}, Partial differential equations and applications. Discrete Contin. Dyn. Syst. vol 10, no. 1-2, (2004),  31--52. 

\bibitem
{BFR80}
 J. Bardina, J. Ferziger, and  W. Reynolds,
  \emph{Improved subgrid scale models for large eddy simulation},
  in Proceedings of the 13th AIAA Conference on Fluid and Plasma Dynamics, (1980).
\bibitem{Bardina}
Y. Cao,  E. M. Lunasin, and E.S. Titi,
{\it Global well-posedness of the three-dimensional
viscous and inviscid simplified Bardina turbulence models}.
Commun. Math. Sci., vol 4, no 4, (2006)  823--848.
\bibitem{Ch-I2001}
V. V. Chepyzhov and A. A. Ilyin,
{\it A note on the fractal dimension of attractors
of dissipative dynamical systems}.
Nonlinear Anal. vol 44
\textbf{44}, (2001), 811--819.
\bibitem{Ch-I}
V. V. Chepyzhov and A. A. Ilyin,
{\it On the fractal dimension of invariant sets; applications to Navier--Stokes equations}.
Discrete Contin. Dyn. Syst.
vol 10, no 1-2, (2004) 117--135.

\bibitem{Ch-V-book}
V. V. Chepyzhov  and M. I. Vishik,
\emph{Attractors for Equations of Mathematical Physics}.
Amer. Math. Soc. Colloq. Publ., vol 49,
Providence, RI: Amer. Math. Soc., 2002.

\bibitem{CIZ}  V.V.Chepyzhov, A.A.Ilyin, S.V.Zelik,
{\it Vanishing viscosity limit for global attractors for the
damped Navier--Stokes system with stress free boundary conditions},
Physica D, vol  376--377, (2018) 31--38.

\bibitem{CF85}
P. Constantin  and C. Foias,
{\it Global Lyapunov exponents, Kaplan--Yorke formulas
and the dimension of the attractors for the 2D
Navier--Stokes equations},
Comm. Pure Appl. Math.
vol 38, (1985) 1--27.
\bibitem{CFT}
P. Constantin, C. Foias and R. Temam,
{\it On the dimension of the attractors in two--dimensional turbulence},
Physica D, vol 30, (1988), 284--296.

\bibitem{DrazinReid}
P. G. Drazin and W. H. Reid.
\emph{Hydrodynamic stability. 2nd ed.}
Cambidge Univ. Press, Cambridge, 2004.
\bibitem{Fef06}
C. Fefferman,  {\it Existence and smoothness of the Navier-Stokes equation,
Millennium Prize Problems,} Clay Math. Inst., Cambridge, MA, (2006),
57--67.
\bibitem{Camassa}
C. Foias, D. D.  Holm, and  E. S. Titi,
 {\it The three dimensional viscous Camassa--Holm equations, and
their relation to the Navier--Stokes equations and turbulence theory},
Jour Dyn and Diff Eqns, vol 14, (2002) 1--35.
\bibitem{FMRT}
C. Foias, O. Manely,  R. Rosa, and R.  Temam,
\emph{Navier--Stokes Equations and Turbulence.}
Cambridge Univ. Press,
Cambridge, 2001.

\bibitem{F95}
U. Frisch, {\it Turbulence. The legacy of A. N. Kolmogorov,} Cambridge University
Press, Cambridge, 1995.
\bibitem{Aro}
A. Haraux,
{\it Two remarks on dissipative hyperbolic problems},
in: Nonlinear Partial Differential
Equations and Their Applications,
College de France Seminar, Vol. VII,
H.\,Brezis, J.L.\,Lions (Eds.), Pitman, London,1985.
\bibitem{henry}
D. Henry, {\it Geometric Theory of Semilinear Parabolic Equations. Lecture
  Notes in Mathematics}, Vol. 840, Springer-Verlag. Berlin-Heidelberg-New York, 1981.

\bibitem
{HLT10}
 M. Holst, E. Lunasin, and G. Tsogtgerel,
  \emph{Analysis of a general family of regularized Navier-Stokes and MHD models},
  J. Nonlinear Sci. vol 20, no  5, (2010) 523--567.



\bibitem{IT1}
A. A. Ilyin  and E. S. Titi,
 {\it Attractors to the two-dimensional Navier--Stokes-$\alpha$
models: an $\alpha$-dependence study},
J. Dynam. Diff. Eqns,
vol 15, (2003) 751--778.
\bibitem{IMT} A. A. Ilyin, A. Miranville, and E. S. Titi, {\it Small viscosity
sharp estimates for the global attractor of the 2-D damped-driven
Navier-Stokes equations},
Commun. Math. Sci. vol 2, (2004) 403--426.


\bibitem{IZLap70} A. A. Ilyin and S. V. Zelik,
{\it Sharp dimension  estimates of the attractor  of the damped  2D Euler-Bardina equations},
\textrm{Partial Differential Equations, Spectral Theory, and Mathematical Physics.
 The Ari Laptev Anniversary Volume of the European Mathematical Society
edited by  T. Weidl. R. Frank, P. Exner, and F. Gesztesy.}




\bibitem{Titi-Varga}
V. K. Kalantarov and E. S. Titi, {\it Global attractors and determining modes for the 3D
Navier--Stokes--Voight equations},
Chin. Ann. Math. vol 30B, no 6, (2009) 697--714.
\bibitem{Lay}
R. Layton and R. Lewandowski, {\it On a well-posed turbulence model}, Discrete Continuous Dyn. Sys. B, vol. 6, (2006) 111--128.

\bibitem{Lad}
O. A. Ladyzhenskaya,
\emph{Attractors for Semigroups and Evolution Equations.}
Leizioni Lincei, Cambridge Univ. Press,
Cambridge, 1991.
\bibitem{L34}
 J. Leray,
  {\it Essai sur le mouvement d’un fluide visqueux emplissant l’space}, Acta Math. vol 63, (1934) 193--248.
\bibitem{LiebJFA}
E. H. Lieb,
{\it An $L^p$ bound for the
    Riesz and Bessel potentials of orthonormal functions},
J. Func. Anal. vol 51, (1983)  159--165.
\bibitem{Lieb}
E. Lieb,
{On characteristic exponents in turbulence,}
Comm. Math. Phys. vol 92, (1984) 473--480.
\bibitem{LT} E. Lieb and W. Thirring,
 {\it Inequalities for the moments of the
eigenvalues of the Schr\"o\-dinger Hamiltonian and their relation to
Sobolev inequalities, Studies in Mathematical Physics}, Essays in
honor  of Valentine Bargmann,
 Princeton University Press,
 Princeton NJ, 269--303 (1976).
\bibitem
{lions}
 J. Lions, \emph{Quelques m\'{e}thodes des probl\`{e}mes aux limites non lin\'{e}aires},
  Doud, Paris, 1969.
\bibitem{PLio}
P. Lions, {\it Mathematical Topics in Fluid Mechanics: Volume 1: Incompressible Models}, Oxford Lecture Series in Mathematics and Its Applications, 1996.

\bibitem{Liu}
V. X. Liu,
{\it A sharp lower bound for the Hausdorff dimension
of the global attractors of the 2D Navier--Stokes equations},
Comm. Math. Phys. vol 158,  (1993) 327--339.
\bibitem{Liu2}
V.X. Liu,  {\it Remarks on the Navier--Stokes equations on
the  two and three dimensional torus},
Comm in PDEs, vol 19, no 5-6, (1994) 873--900.
\bibitem{Lopes}
M. Lopes Filho, H. Nussenzveig Lopes, E. Titi, A. Zang.
{\it  Convergence of the 2D Euler-$\alpha$  to Euler equations in the Dirichlet case: indifference to boundary layers}, Phys. D, vol 292-293, (2015) 51--61.
\bibitem{MirZel}
A. Miranville and S. Zelik,
{\it Attractors for dissipative partial differential equations in bounded and unbounded domains},
 In: Handbook of differential equations: evolutionary equations. Vol. IV, 103–200,
  Handb. Differ. Equ., Elsevier/North-Holland, Amsterdam, 2008.
\bibitem{rosa}  I. Moise, R. Rosa, and X. Wang.  {\it Attractors for non-compact semigroups via
 energy equations}, Nonlinearity, vol. 11, no. 5 (1998), 1369--1393.
\bibitem
{OT07}
 E. Olson and E. Titi,
 \emph{Viscosity versus vorticity stretching: global well-posedness
  for a family of Navier-Stokes-$\alpha$-like models},
   Nonlinear Anal. vol 66, no  11, (2007) 2427--2458.

\bibitem{Osk}
A. Oskolkov, {\it The uniqueness and solvability in the large of boundary value problems for the equations of motion of aqueous solutions of polymers}, Zap. Nauchn. Sem. LOMI, vol. 38, (1973) 98--136.
\bibitem{Ped}
J. Pedlosky, {\it Geophysical Fluid Dynamics},
Springer, New York, 1979.
\bibitem{traceSimon}
B. Simon,
\emph{Trace Ideals and Their Applications, \rm 2nd ed.}
Amer. Math. Soc., Providence RI, 2005.
\bibitem{Squire}
H.B. Squire,
{\it On stability of three-dimensional disturbances of
viscous flow between parallel walls},
Proc. R. Soc. Lond. Ser. A.
vol 142(847), (1993) 621--628.
\bibitem{S-W}
E. M. Stein and G. Weiss,
\emph{Introduction to Fourier analysis on Euclidean spaces.}
Princeton University Press,
 Princeton NJ, 1972.
\bibitem{tao}
T. Tao, {\it Finite time blowup for an averaged three-dimensional Navier-Stokes equation}, J. Amer. Math. Soc., vol 29 (2016), 601--674
\bibitem{T}
R.Temam,
\emph{Infinite Dimensional Dynamical Systems in
Mechanics and Physics, \rm 2nd ed.}
Sprin\-ger-Ver\-lag, New York  1997.
\bibitem
{T95}
R. Temam, \emph{Navier-Stokes equations and nonlinear functional analysis},
vol. 66, Siam, 1995.
\bibitem{Watson}
G. N. Watson,
\emph{A Treatise on the Theory of Bessel Functions, \rm 2nd ed.}
Cambridge University Press,
Cambridge, 1995.


\end{thebibliography}
\end{document}